\documentclass[11pt]{article}
\usepackage[utf8]{inputenc}
\usepackage[english]{babel}
\usepackage[T1]{fontenc}
\usepackage{float}
\usepackage{listings}
\usepackage{amsmath}
\usepackage{amsfonts}
\usepackage{amssymb}
\usepackage{amsthm}
\usepackage{enumerate}
\usepackage{anysize}
\usepackage{graphicx}
\usepackage{xcolor}
\usepackage{color}
\usepackage{fancyhdr}
\usepackage{makeidx}
\usepackage{hyperref}
\usepackage{appendix}
\usepackage{tikz}
\usepackage[matrix,arrow]{xy}
\usetikzlibrary{arrows} 
\usepackage{url}
\usepackage[square,numbers]{natbib}
\usepackage{subcaption}
\usepackage{stackrel}
\marginsize{2cm}{2cm}{0cm}{2cm}

\definecolor
{gray75}{gray}{0.75}
\definecolor
{gray85}{gray}{0.85}
\numberwithin{equation}{section}

\newtheorem{teo}{Theorem}[section]
\newtheorem{lem}{Lemma}[section]

\newtheorem{obs}{Remark}[section]
\newtheorem{col}{Corollary}[section]

\begin{document}
\title{\textbf{A Glioblastoma PDE-ODE model including chemotaxis and vasculature}}
	\author{ A. Fernández-Romero$^{2}$, F. Guillén-González$^{2}$\footnote{ORCID: 0000-0001-5539-5888}, A. Suárez$^{1\;2}$\footnote{ORCID: 0000-0002-3701-6204}.\\ \small{$^{1}$Corresponding author.}\\
	\small{$^{2}$Dpto. Ecuaciones Diferenciales y Análisis Numérico,}\\
	\small{Facultad de Matemáticas, Universidad de Sevilla. Sevilla, Spain.}\\
	\small{\texttt{afernandez61@us.es, guillen@us.es, suarez@us.es}}	
}
\date{}
\maketitle
\section*{\centering{\textbf{Abstract}}}
In this work we analyse a PDE-ODE problem modelling the evolution of a Glioblastoma, which includes chemotaxis term directed to vasculature. First, we obtain some a priori estimates for the (possible) solutions of the model. In particular, under some conditions on the parameters, we obtain that the system does not develop blow-up at finite time. In addition, we design a fully discrete finite element scheme for the model which preserves some pointwise estimates of the continuous problem. Later, we make an adimensional study in order to reduce the number of parameters. Finally, we detect the main parameters determining different width of the ring formed by proliferative and necrotic cells and different regular/irregular behaviour of the tumor surface.
\\
\\
\textbf{Mathematics Subject Classification.} $35\text{A}01,\;35\text{B}40,\;35\text{M}10,\;35\text{Q}92,\;47\text{J}35,\;92\text{B}05$\\
\textbf{Keywords:} Glioblastoma, Chemotaxis, PDE-ODE system, Numerical scheme.
\\
\\
The authors were supported by PGC2018-098308-B-I00 (MCI/AEI/FEDER, UE).

\section{Introduction}\label{introducion}
Among the group of brain tumors, the Glioblastoma (GBM) is the most aggressive form with a survival of a little more than one year \cite{Ostrom_2014}. Moreover, GBM differs from many solid tumors in the sense that they grow infiltratively into the brain tissue, there exists an important presence of necrosis and they produce a high proliferation tumor cells. For all these reasons, GBM is one of the cancer types with more interest in the mathematical oncology community (see \cite{Alfonso_2017, Baldock_2013, Protopapa_2018} and references therein).
\\

Some studies about the morphology of GBM are based in the magnetic resonance images (MRI) in order to obtain results related to prognosis and survival (see \cite{Molina_2019,Julian_2016,Victor_2019,Victor_2018}). Specifically, Molab\footnote{\href{Molab}{http://matematicas.uclm.es/molab/}} group classifies the GBM depending on the width of the tumor ring and/or the tumor surface regularity (see \cite{Julian_2016, Victor_2018} respectively). The study of \cite{Julian_2016} concludes that tumors with slim ring have better prognostic, specifically $7$ months of more survival than tumors with thick ring. In \cite{Victor_2018}, the survival of patients in relation to the surface growth, regular or irregular, of the GBM, show that tumors with a regular surface have better prognostic, more than $5$ moths of survival, than tumor with irregular surface. 
\\

In \cite{Jan_2014}, the authors use the Fisher-Kolmogorov equation to reproduce the infiltrative characteristic of the GBM. However, more complex mathematical models are also built to simulate phenomena such that the tumor ring and the regularity surface of the GBM. One model appears in \cite{Victor_2020} where the tumor ring is studied by a PDE-ODE system of two equations (proliferative tumor and necrosis). In \cite{Romero_2021,Romero2_2020}, the authors present a PDE-ODE system with three equations (proliferative tumor, necrosis and vasculature) which is able to capture different behaviours of tumor ring and regularity surface of the GBM via a nonlinear diffusion tumor increasing with vasculature.
\\
%

In this paper, we present a PDE-ODE system, also with three equations (tumor, necrosis and vasculature)  and we study the  biological behaviours of the GBM such as the tumor ring volume, studied in \cite{Victor_2020,Julian_2016}, and the regularity surface considered in \cite{Victor_2018}. Unlike the system considered in  \cite{Romero_2021,Romero2_2020}, we have included a chemotaxis term. This term has been already introduced to model the movement of  some populations towards a higher concentration of the chemical substance or another living organism, see for instance the reviews given in \cite{Bellomo,EC,HP,Hor,Per} and the references herein. Specifically, in this paper, we have included the chemotaxis term modelling the movement of tumor to vasculature.
\\

Some previous chemotactic PDE-ODE models have been extensively studied in the literature, see for instance \cite{Chaplain_2000,Sleeman_1997,Stevens_2000,Othmer_1997} where the authors model the cells movement with a parabolic-ODE system. Specifically, in \cite{Othmer_1997} a system of
PDEs is considered using a probabilistic framework of reinforced random walks. The authors analyse various combinations of taxis and local dynamics giving examples of aggregation, blow-up and collapse. Later, in \cite{Sleeman_1997}, some analytical and numerical results which support the numerical observations of \cite{Othmer_1997} are presented using a similar model than in \cite{Othmer_1997}. Moreover, in \cite{Anderson_1998,Chaplain_2000} a model of tumor inducing angiogenesis is proposed consisting of a equation with chemotaxis and haptotaxis term, and two nonlinear ODEs. Finally, in \cite{Stevens_2000} a  stochastic system related to bacteria and particles of chemical substances is discussed where the position of each particle is described by a equation of a chemotaxis system.
\\
%

Several works such as \cite{cris,Tao_2010,Tao_2009,Tao_2012} have shown existence results for systems of three differential equations modelling cancer invasion. In \cite{Tao_2010} the global existence and boundedness of solution for a parabolic-parabolic-ODE system with nonlinear density-dependent chemotaxis and haptotaxis and logistic source is deduced. Furthermore, in \cite{Tao_2009}, the authors have proved global existence of solutions for a parabolic-elliptic-ODE system with chemotaxis, haptotaxis and logistic growth. The study of existence of solutions for the chemotaxis and haptotaxis model with nonlinear diffusion is presented in \cite{Tao_2012}. The global existence of solution and its asymptotic behaviour are studied in \cite{cris} for a parabolic-parabolic-ODE system modelling the cells invasion process.
\\
%

Recently, a PDE-ODE model with chemotaxis is studied in \cite{Wang_2020} obtaining asymptotic stability results using a proper transformation and energy estimates. Another PDE-ODE with chemotaxis problem is considered in \cite{Tello_2020}, see also \cite{MTello},  modelling the evolution of biological species and they obtain analytical results concerning the bifurcation of constant steady states and global existence of solutions for a range of initial data. In \cite{FriTello} a parabolic-ODE problem is analysed, and it is shown that, under several conditions, any stationary solution is locally stable.
\\

In this paper, we investigate the following parabolic PDE-ODE system in $\left(0,T_f\right)\times\Omega$ ($\Omega\subseteq\mathbb{R}^3$ is a bounded and regular domain and $T_f>0$ corresponds to the final time)


\begin{equation}\label{probOriginal}
\left\{\begin{array}{ccl}
\dfrac{\partial T}{\partial t} -\underbrace{\nu\; \Delta T}_{\text{Diffusion}}+\underbrace{\kappa\;\nabla\cdot\left(T\;\nabla\Phi\right)}_{\text{Chemotaxis}}& = & f_1\left(T,N,\Phi\right)\\

&&\\
\dfrac{\partial N}{\partial t} & = &  f_2\left(T,\Phi\right)\\

&&\\
\dfrac{\partial \Phi}{\partial t} & = &  f_3\left(T,N,\Phi\right)\\

\end{array}\right.
\end{equation}
\\
endowed with non-flux boundary condition on the boundary $\partial\Omega$
\begin{equation}\label{condifronte}
\left(-\nu\; \nabla T +\kappa\; T\; \nabla \Phi \right) \cdot n =0
\end{equation} 
where $n$ is the outward unit normal vector to $\partial\Omega$ and initial conditions at time $t=0$:
\begin{equation}\label{condinicio}
T\left(0,\cdot\right)=T_0,\;N\left(0,\cdot\right)=N_0,\;\Phi\left(0,\cdot\right)=\Phi_0
\;\;\text{in}\;\;\Omega.
\end{equation}

Here,  
$T(t,x), N(t,x)$ and $\Phi(t,x)$ represent the tumor and necrotic densities and the vasculature concentration at the point $x\in\Omega$ and time $t>0$, respectively. 
\\

%
The nonlinear reactions functions $f_i\;:\mathbb{R}^3\rightarrow\mathbb{R}$ for $i=1,2,3$ have the following form

\begin{equation}\label{funciones}
\left\{\begin{array}{ccl}
f_1\left(T,N,\Phi\right) &:=& \underbrace{\rho\;P\left(\Phi,T\right) \;T\left(1-\dfrac{T+N+\Phi}{K}\right)}_{\text{Tumor growth }}-\underbrace{\alpha\;S(\Phi,T)\;T}_{\text{Hypoxia}}\\

\\
f_2\left(T,N,\Phi\right) &:=&  \alpha\;S(\Phi,T)\;T+\delta\;Q\left(\Phi,T\right)\;\Phi\\

\\
f_3\left(T,N,\Phi\right) &:=&\underbrace{\gamma\;R\left(\Phi,T\right)\Phi\left(1-\dfrac{T+N+\Phi}{K}\right)}_{\text{Vasculature growth }}-\underbrace{\delta\;Q\left(\Phi,T\right)\;\Phi}_{\substack{\text{Vascular destruction}\\ \text{by the tumor}}}
\end{array}\right.
\end{equation}

The parameters in $\left(\ref{probOriginal}\right)$ have the following description \cite{Klank_2018,Alicia_2015, Alicia_2012}:

\begin{table}[H]
	\centering
	\begin{tabular}{c|c|c}
		\textbf{Variable} & \textbf{Description} & \textbf{Value} \\
		\hline
		$\nu$	& Speed diffusion & $\dfrac{\text{cm}^2}{\text{sec}}$     \\
		\hline
		$\kappa$	& Speed chemotaxis & $\dfrac{\text{cm}^2}{\text{sec}\cdot\text{density}}$     \\
		\hline
		$\rho$	& Tumor proliferation rate  &  $\text{day}^{-1}$    \\
		\hline
		$\alpha$ & Hypoxic death rate  & $\text{day}^{-1}$    \\
		\hline
		$\gamma$	& Vasculature proliferation rate  &  $\text{day}^{-1}$    \\
		\hline
		$\delta$	& Vasculature destruction by tumor   &  $\text{day}^{-1}$    \\
		\hline
		$K$	& Carrying capacity & $\text{cell}/\text{cm}^3$ \\
	\end{tabular}
	\caption{\label{Table1} Parameters.}
\end{table}

The functions $P\left(\Phi,T\right)$, $S\left(\Phi,T\right)$, $R\left(\Phi,T\right)$ and $Q\left(\Phi,T\right)$ appearing in $\left(\ref{funciones}\right)$ are adimensional factors with the following biological meaning:
\begin{enumerate}
	
	\item The tumor growth cells need space and a well amount of nutrients to grow. If this amount of nutrients per cell is suitable, the proliferation of tumor cells will occur. Hence, we introduce the tumor proliferation factor $P\left(\Phi,T\right)$ in $f_1$ as a volume fraction of the vasculature.
	
	\item We consider the hypoxia as a decreasing term due to lack of vasculature. Hence, low vasculature produces more tumor destruction. Therefore, the factor $S\left(\Phi,T\right)$ must be a volume fraction of the lack of vasculature.
	
	\item The vasculature growth factor $R\left(\Phi,T\right)$ will depend on the amount of tumor and 
	the vasculature does not grow without tumor. 
	Thus, $R\left(\Phi,T\right)$ will be a volume fraction of tumor.
	
	\item The destruction of vasculature will increase with tumor and there will not be vascular destruction without tumor. In consequence, $Q\left(\Phi,T\right)$ will be a volume fraction of tumor.
\end{enumerate} 

Thus, these factor functions $P\left(\Phi,T\right)$, $S\left(\Phi,T\right)$, $R\left(\Phi,T\right)$ and $Q\left(\Phi,T\right)$ must satisfy the following modelling conditions:

\begin{equation}\label{PQ_menor1}
0\leq P\left(\Phi,T\right),\;S\left(\Phi,T\right),\;Q\left(\Phi,T\right),\;R\left(\Phi,T\right)\leq1\quad\forall\left(T,\Phi\right)\in\mathbb{R}^2,
\end{equation}
and, 
\begin{equation}\label{p}
	 	P\left(\Phi,T\right)=0 \text{ for } \Phi=0 \text{ and } P\left(\Phi,T\right) \text{ increases if } \Phi \text{ increases,}
\end{equation}
	 \begin{equation}\label{s}
	S\left(\Phi,T\right) \text{ increases if } \Phi \text{ decreases,}
	 \end{equation}	 
	 \begin{equation}\label{r}
	 R\left(\Phi,T\right)=0 \text{ for } T=0 \text{ and } R\left(\Phi,T\right) \text{ increases if } T \text{ increases (at least for } T\le K\text{),}
	 \end{equation}
\begin{equation}\label{q}
Q\left(\Phi,T\right)=0 \text{ for } T=0 \text{ and } Q\left(\Phi,T\right) \text{ increases if } T \text{ increases.}
\end{equation}

We assume along the paper the following assumptions on the initial data
\begin{equation}\label{hipotesis0}
0\leq T_0(x),N_0(x),\Phi_0(x)\leq K,\;\; \text{  a.e.}\;x\in\Omega.
\end{equation}

In order to obtain some estimates of the solutions of  $\left(\ref{probOriginal}\right)$-$\left(\ref{condinicio}\right)$ (see $\left(\ref{cotas}\right)$), we define the following truncated system of $\left(\ref{probOriginal}\right)$:
\begin{equation}\label{problin}
\left\{\begin{array}{ccl}
\dfrac{\partial T}{\partial t} -\nu\; \Delta T+\kappa\;\nabla\cdot\left(T_+\;\nabla\Phi\right)& = & f_1\left(T_+,N_+,\Phi_+^{K}\right)\\
&&\\
\dfrac{\partial N}{\partial t} & = &  f_2\left(T_+,\Phi_+\right)\\
&&\\
\dfrac{\partial \Phi}{\partial t} & = & f_3\left(T_+,N_+,\Phi_+^{K}\right)\\
\end{array}\right.
\end{equation}
\\
subject to $\left(\ref{condifronte}\right)$ and $\left(\ref{condinicio}\right)$.
We have denoted $\Phi_+^{K}=\min\left\{K,\max\left\{0,\Phi\right\}\right\}$ and $T_+=\max\left\{0,T\right\}$ and the same for $N_+$ and $\Phi_+$. 
\\

The main contributions of this work are the following:

\begin{enumerate}
	\item \begin{teo}[A priori estimates]\label{teo_estimaciones}

\begin{enumerate}[a)]
	\item Any regular enough solution $\left(T,N,\Phi\right)$ of the truncated problem $\left(\ref{problin}\right)$-$\left(\ref{condinicio}\right)$ satisfies:
	$$0\leq \Phi\leq K,\;\;T\geq 0\;\;\text{and}\;\; N\geq 0,\;\; \text{a.e. in}\;\; \left(0,T_f\right)\times\Omega$$
	and
	$$T,\; N\text{ are bounded in } L^\infty\left(0,T_f;L^{1}\left(\Omega\right)\right).$$
	\item Assuming that there exists a constant $C_1>0$ such that
	 \begin{equation}\label{P_menor_R_Phi}
	C_1\;P\left(\Phi,T\right)\geq R\left(\Phi,T\right)\;\Phi,\quad \forall \;0\le \Phi\le K\text{, and } T\ge 0
	\end{equation} 
	and 
	\begin{equation}\label{rho_mayor_p_gamma}
	\rho\geq \dfrac{\kappa}{\nu}\;\gamma\;C_1,
	\end{equation}
	then
	$$T,\; N\text{ are bounded in } L^\infty\left(0,T_f;L^{\infty}\left(\Omega\right)\right).$$
	\item  Assuming additionally that there exist constants $C_i>0$ for $i=2,3,4$ such that for all $0\le\Phi\le K$ and $T\ge 0$,
	
	\begin{equation}\label{der_RPhi_T_L_infinito}
	\Big|\dfrac{\partial\left(R\left(\Phi,T\right)\;\Phi\right)}{\partial\;\Phi}\Big|,\;\Big|\dfrac{\partial\left(R\left(\Phi,T\right)\;\Phi\right)}{\partial\;T}\Big|\leq C_2,
	\end{equation}
	
	\begin{equation}\label{der_QPhi_L_infinito}
	\Big|\dfrac{\partial\left(Q\left(\Phi,T\right)\;\Phi\right)}{\partial\;\Phi}\Big|,\;\Big|\dfrac{\partial\left(Q\left(\Phi,T\right)\;\Phi\right)}{\partial\;T}
	\Big|\leq C_3
	\end{equation}
	
	and
	\begin{equation}\label{parciales_S_Q_T_L_infitnito}
	\Big|\dfrac{\partial\left(S\left(\Phi,T\right)\;T\right)}{\partial\;\Phi}\Big|,\;\Big|\dfrac{\partial\left(S\left(\Phi,T\right)\;T\right)}{\partial\;T}\Big|\leq C_4,
	\end{equation} 
	then

		 $$\nabla N,\;\nabla \Phi \text{ are bounded in } L^\infty\left(0,T_f;L^{2}\left(\Omega\right)\right),$$
		 and
		 $$\nabla T \text{ is bounded in } L^2\left(0,T_f;L^{2}\left(\Omega\right)\right).$$
	
\end{enumerate}
	\end{teo} 
	By Theorem $\ref{teo_estimaciones}$ a), for any $\left(T,N,\Phi\right)$ solution of $\left(\ref{problin}\right)$, we deduce that $T_+=T$, $N_+=N$ and $\Phi_+^{K}=\Phi$ and then, $f_i\left(T_+,N_+,\Phi_+^{K}\right)=f_i\left(T,N,\Phi\right)$ for $i=1,3$ and $f_2\left(T_+,\Phi_+^K\right)=f_2\left(T,\Phi\right)$. Hence, we obtain the following crucial corollary:
\begin{col}\label{solTrunc-solOrigin}
	If $\left(T,N,\Phi\right)$ is a solution of the truncated problem $\left(\ref{problin}\right)$, then $\left(T,N,\Phi\right)$ is also a solution of
	$\left(\ref{probOriginal}\right)$-$\left(\ref{condinicio}\right)$ and $\left(T,N,\Phi\right)$ satisfies the  estimates of Theorem $\ref{teo_estimaciones}$.
\end{col}

The existence of solutions of problem $\left(\ref{problin}\right)$ is out of the scope of this paper. It is an interesting open problem that could be treated in a forthcoming paper.
	
	\item  In Section $\ref{esquema_espacio}$, we design a Finite Element numerical scheme, computing $\left(\Phi_k^h, T_k^h, N_k^h\right)$ as an approximation of $\left(\Phi(t_k,\cdot), T(t_k,\cdot), N(t_k,\cdot)\right)$ where $t_k$ is a partition of the time interval $\left(0,T_f\right)$ and $h$ is the mesh size. To build the scheme, we will use the change of variable in the PDE equation with chemotaxis, $T=e^{\frac{\kappa}{\nu}\;\Phi}\;u$, similar to the used in \cite{Corrias_2004,Anderson_2015,Mariya_2010}, in order to obtain an equivalent system with diffusion for the new variable $u$.
	\begin{teo}[Discrete version of Theorem $\ref{teo_estimaciones}$ a)]\label{teo_estimaciones_discreto}
		Scheme $\left(\ref{eqT_space}\right)$-$\left(\ref{f3_space}\right)$ has a unique solution satisfying the first pointwise estimates of Theorem $\ref{teo_estimaciones}$ a), these are: 
		\begin{equation}\label{estimaciones_discretas}
		0\leq \Phi_h^k\leq K,\;\;T_h^k\geq 0\;\;\text{and}\;\; N_h^k\geq 0,\;\; \text{in}\;\; \Omega.
		\end{equation}
	\end{teo}

    The design of a numerical scheme preserving the whole estimates of Theorem $\ref{teo_estimaciones}$, and not only the estimates $\left(\ref{estimaciones_discretas}\right)$, remains as an open problem.
	
	\item A parametric study through numerical simulations is made in order to detect different behaviours for the ring width and the regularity of the surface of the tumor.
\end{enumerate}

The outline of the paper is as follows. In Section $\ref{estimates}$, we prove Theorem $\ref{teo_estimaciones}$. In Section $\ref{esquema_espacio}$ we build a numerical scheme which preserves the a priori estimates of the continuous model given in Theorem $\ref{teo_estimaciones}$ a). Later, in Section $\ref{adimensionalizacion}$, we show a possible example of the dimensionless reaction functions of the system satisfying the hypotheses given in $\left(\ref{p}\right)$-$\left(\ref{q}\right)$ and $\left(\ref{P_menor_R_Phi}\right)$-$\left(\ref{parciales_S_Q_T_L_infitnito}\right)$ and we make an adimensionalization of the model. Section $\ref{numerical_simulations}$ is dedicated to show, by means of some numerical simulations, the different behaviour of the ring width-volume and the regularity surface with respect to the dimensionless parameters. Finally, the more technical part of the proof of Theorem $\ref{teo_estimaciones}$ b), obtained via an Alikakos' argument, is given in an Appendix.

\section{A priori estimates of the solutions of  $\boldsymbol{\left(\ref{problin}\right)$-$\left(\ref{condinicio}\right)}$}\label{estimates}

\subsection{Proof of Theorem $\ref{teo_estimaciones}$ a)}
\begin{lem}
	\label{estimaciones1} 
	Any solution $(T,N,\Phi)$ of the truncated problem $\left(\ref{problin}\right)$ satisfy the following pointwise estimates:
	
	\begin{equation}\label{cotas}
	0\leq \Phi\leq K,\;\;T\geq 0\;\;\text{and}\;\; N\geq 0,\;\; \text{a.e. in}\;\; \left(0,T_f\right)\times\Omega.\end{equation}

\end{lem}
\begin{proof}
	Let $\left(T,N,\Phi\right)$ be a solution of $\left(\ref{problin}\right)$. Since one can rewrite  $f_1(T_+,N_+,\Phi_+^{K})=T_+\;\widetilde{f_1}(T_+,N_+,\Phi_+^{K})$, multiplying the first equation of $\left(\ref{problin}\right)$ by $T_-=\min\left\{T,0\right\}$ and integrating in $\Omega$, we get
	
	$$\dfrac{1}{2}\dfrac{d}{dt}\int_{\Omega}(T_-)^2\;dx+\nu\int_{\Omega}\mid\nabla T_-\mid^2
	=\int_{\Omega}T_-\;T_+\;\widetilde{f_1}\left(T_+,N_+,\Phi_+^{K}\right)\;dx=0,\quad\text{  a.e. in}\;\left(0,T_f\right).$$
	
	Hence, since $T_-\left(0,x\right)=0$, then $T_{-}\left(t,x\right)=0$ a.e. $\left(t,x\right)\in\left(0,T_f\right)\times\Omega$.
	We repeat the same argument for the other two equations of $\left(\ref{problin}\right)$ using now that 
	$$\Phi_-\;f_3\left(T_+,N_+,\Phi_+^{K}\right)=0\;\;\text{and}\;\;N_-\;f_2\left(T_+,\Phi_+\right)\leq0.$$
	To obtain the upper bound $\Phi\le K$, we multiply the third equation of $\left(\ref{problin}\right)$ by $\left(\Phi-K\right)_+=\max\left\{0,\Phi-K\right\}$ and integrate in $\Omega$,
	
	$$\dfrac{1}{2}\dfrac{d}{dt}\int_{\Omega}\left(\left(\Phi-K\right)_+\right)^2\;dx=\int_{\Omega}f_3\left(T_+,N_+,\Phi_+^{K}\right)\left(\Phi-K\right)_+\;dx,\quad\text{a.e. in}\;\left(0,T_f\right).$$
	Since $f_3(T_+,N_+,\Phi_+^{K})\leq \gamma\; \Phi_+^K(1-\frac{\Phi_+^{K}}{K})$, then $f_3(T_+,N_+,\Phi_+^{K})(\Phi-K)_+\le0$. As $\left(\Phi\left(0,x\right)-K\right)_+=0$, then $\left(\Phi\left(t,x\right)-K\right)_+=0$  a.e. $\left(t,x\right)\in\left(0,T_f\right)\times\Omega$. 

\end{proof}	

%

\begin{lem}\label{estimaciones2} 
	Any solution of $(T,N,\Phi)$ satisfies the estimates:

	 \begin{equation}\label{Linf_L1_T}
	\| T\|_{{L}^\infty\left(0,T_f;L^1\left(\Omega\right)\right)}+\| \sqrt{P\left(\Phi,T\right)}\;T\|_{{L}^2\left(0,T_f;L^2\left(\Omega\right)\right)}\leq C\left(\rho,K,|\Omega|,T_f\right),
	\end{equation}
	\\
	\begin{equation}\label{Linf_L1_N}
	\| N\|_{{L}^\infty\left(0,T_f;L^1\left(\Omega\right)\right)}\leq C\left(\rho,\alpha,\delta,K,|\Omega|,T_f\right).
	\end{equation}
\end{lem}

\begin{proof}	
	Let $\left(T,N,\Phi\right)$ be a solution of $\left(\ref{problin}\right)$. Integrating in $\Omega$ the first equation of $\left(\ref{problin}\right)$ and using that $P\left(\Phi,T\right),\;S\left(\Phi,T\right)\geq0$, we obtain that
	
	$$\displaystyle\dfrac{d}{dt}\int_{\Omega}T\;dx=\int_{\Omega}\rho\;P\left(\Phi,T\right)\;T\;dx-\int_{\Omega}\rho\;P\left(\Phi,T\right)\dfrac{T^2}{K}\;dx-\int_{\Omega}\rho\;\underbrace{P\left(\Phi,T\right)\;T\;\dfrac{N+\Phi}{K}}_{\geq0}\;dx-$$
	$$-\int_{\Omega}\alpha\;\underbrace{S\left(\Phi,T\right)T}_{\geq0}\;dx\leq\int_{\Omega}\rho\;P\left(\Phi,T\right)\;T\;dx-\dfrac{1}{K}\int_{\Omega}\rho\;P\left(\Phi,T\right)T^2\;dx.$$
	
	Thus,
	
	$$\displaystyle\dfrac{d}{dt}\int_{\Omega}T\;dx+\dfrac{1}{K}\int_{\Omega}\rho\;P\left(\Phi,T\right)T^2\;dx\leq\int_{\Omega}\rho\;P\left(\Phi,T\right)\;T\;dx.$$
	
	Rewriting $P\left(\Phi,T\right)\;T=\sqrt{P\left(\Phi,T\right)}\sqrt{P\left(\Phi,T\right)}\;T$ and applying Young's inequality for the right side, we get,
	
	$$\displaystyle\dfrac{d}{dt}\int_{\Omega}T\;dx+\dfrac{1}{K}\int_{\Omega}\rho\;P\left(\Phi,T\right)T^2\;dx\leq\rho\left(\dfrac{1}{2\;K}\int_{\Omega}P\left(\Phi,T\right)T^2\;dx+\dfrac{K}{2}\int_{\Omega}P\left(\Phi,T\right)\right).$$
	
	Hence, using that $P\left(\Phi,T\right)\leq1$, we conclude that
	
	$$\displaystyle\dfrac{d}{dt}\int_{\Omega}T\;dx+\dfrac{\rho}{2\;K}\int_{\Omega}P\left(\Phi,T\right)T^2\;dx\leq\dfrac{\rho\;K}{2}|\Omega|.$$
	
	Integrating in $\left(0,t\right)$ for $0<t\leq T_f$, we obtain that
	
	$$\|T\left(t,\cdot\right)\|_{L^1\left(\Omega\right)}+\dfrac{\rho}{2\;K}\int_{0}^{t}\int_{\Omega}P\left(\Phi,T\right)T^2\;dx\;dt\leq T_f\;\dfrac{\rho\;K}{2}|\Omega|,\quad \forall t\in\left(0,T_f\right)$$
	whence we deduce $\left(\ref{Linf_L1_T}\right)$.
	\\
	
	To prove $\left(\ref{Linf_L1_N}\right)$, we integrate the second equation of $\left(\ref{problin}\right)$ in $\Omega\times\left(0,t\right)$, with $0<t\leq T_f$,
	
	$$\|N\left(t,\cdot\right)\|_{L^1\left(\Omega\right)}\leq\alpha\;\int_{0}^{t}\int_{\Omega}T\;dx\;dt+\delta\;\int_{0}^{t}\int_{\Omega}\Phi\;dx\;dt$$
	where we have used $\left(\ref{PQ_menor1}\right)$. Thus, using that $\Phi\leq K$ and the bound obtained for $T$ in $\left(\ref{Linf_L1_T}\right)$, we get $\left(\ref{Linf_L1_N}\right)$.

\end{proof}

\subsection{Proof of Theorem $\ref{teo_estimaciones}$ b)}
In order to obtain the $L^\infty$ estimate for $T$, firstly we make a change of variable such that we rewrite the diffusion term and chemotaxis term as an unique diffusion term depending on the new variable. In fact, we consider:

\begin{equation}\label{cambio_variable}
w=\log\left(T\right)-\chi\;\Phi\Leftrightarrow T=e^w\;e^{\chi\;\Phi}=e^{\chi\;\Phi}\;u
\end{equation}
with $u=e^{w}$ and $\chi=\dfrac{\kappa}{\nu}$. 
\\

Thus, the first equation of $\left(\ref{probOriginal}\right)$ changes to

\begin{equation}\label{eqTcambio}
\left(e^{\chi\;\Phi}\;u\right)_t-\nu\nabla\cdot\left(e^{\chi\;\Phi}\;\nabla\;u\right)=f_1\left(e^{\chi\;\Phi}\;u,N,\Phi\right)
\end{equation}
and the boundary condition $\left(\ref{condifronte}\right)$ to
\begin{equation}\label{condifronte2}
\nabla u \; \cdot n =0.
\end{equation} 
	\begin{lem}[Proof of Theorem $\ref{teo_estimaciones}$ b)]\label{estimaciones3} 
	Assume $\left(\ref{P_menor_R_Phi}\right)$ and $\left(\ref{rho_mayor_p_gamma}\right)$.
	Then, given any solution $\left(T,N,\Phi\right)$ of $\left(\ref{problin}\right)$, 
	it holds that $u$ is bounded in $L^\infty\left(0,T_f;L^{\infty}\left(\Omega\right)\right)$ and $\nabla u$ is bounded in $L^2\left(0,T_f;L^{2}\left(\Omega\right)\right)$. Moreover, $T$ and $N$ are bounded in $L^\infty(0,T_f;L^\infty(\Omega))$.
	
\end{lem}
\begin{proof}
	To obtain the $L^\infty$ estimates for $T$ and $N$, taking into account the $L^\infty$ estimates for $\Phi$, it suffices that $u$ be $L^\infty$. The proof of $u$ is based in $L^p$ estimates with an Alikakos' argument. Let $\left(T,N,\Phi\right)$ be a solution of $\left(\ref{problin}\right)$. We multiply $\left(\ref{eqTcambio}\right)$ by $u^{p-1}$ (for any $p\geq2$), and analyse term by term:
	
	\begin{itemize}
		\item Time derivative term:
		\begin{equation}\label{derivada_t_cambio}
		\left(e^{\chi\;\Phi}\;u\right)_t\;u^{p-1}=\chi\;\Phi_t\;e^{\chi\;\Phi}\;u^p+\dfrac{1}{p}e^{\chi\;\Phi}\left(u^p\right)_t
		\end{equation}
		and the second term of the right side of $\left(\ref{derivada_t_cambio}\right)$ can be expressed as
		\begin{equation}\label{derivada_t_cambio2}
		\dfrac{1}{p}e^{\chi\;\Phi}\left(u^p\right)_t=\dfrac{1}{p}\left(e^{\chi\;\Phi}\;u^p\right)_t-\dfrac{\chi}{p}\;e^{\chi\;\Phi}\;u^p\;\Phi_t.
		\end{equation}
		Hence, from $\left(\ref{derivada_t_cambio}\right)$ and $\left(\ref{derivada_t_cambio2}\right)$,
		\begin{equation}\label{derivada_t_cambio3}
		\left(e^{\chi\;\Phi}\;u\right)_t\;u^{p-1}=\dfrac{1}{p}\left(e^{\chi\;\Phi}\;u^p\right)_t+\dfrac{p-1}{p}\;\chi\;\Phi_t\;e^{\chi\;\Phi}\;u^p.
		\end{equation}
		\item Nonlinear diffusion term:
		\begin{equation}\label{difusion_cambio}
		\begin{array}{c}
		-\nu\nabla\cdot\left(e^{\chi\;\Phi}\;\nabla\;u\right)\;u^{p-1}=-\nu\;\nabla\cdot\left(e^{\chi\;\Phi}\left(\nabla\;u\right)u^{p-1}\right)+\nu\;e^{\chi\;\Phi}\left(p-1\right)\;u^{p-2}\mid\nabla\;u\mid^2\\
		\\
		=-\nu\;\nabla\cdot\left(e^{\chi\;\Phi}\left(\nabla\;u\right)u^{p-1}\right)+\nu\;e^{\chi\;\Phi}\left(p-1\right)\;\dfrac{4}{p^2} \mid\nabla (u^{p/2})\mid^2.
		\end{array}
		\end{equation}
		\item Reaction term:
		\begin{equation}\label{reacion_cambio}
		f_1\left(e^{\chi\;\Phi}\;u,N,\Phi\right)\;u^{p-1}=\rho\;P\left(\Phi,T\right)\;e^{\chi\;\Phi}\;u^p\left(1-\dfrac{e^{\chi\;\Phi}\;u+N+\Phi}{K}\right)-\alpha\;S\left(\Phi,T\right)e^{\chi\;\Phi}\;u^p.
		\end{equation}
	\end{itemize}

	Rewriting in $\left(\ref{derivada_t_cambio3}\right)$ the function $\Phi_t$ as $f_3\left(e^{\chi\;\Phi}\;u,N,\Phi\right)$ and adding $\left(\ref{derivada_t_cambio3}\right)$, $\left(\ref{difusion_cambio}\right)$ and $\left(\ref{reacion_cambio}\right)$, we get:

	\begin{equation}\label{L2H1_prueba1}
	\begin{array}{c}
	\dfrac{1}{p}\left(e^{\chi\;\Phi}\;u^p\right)_t-\nu\;\nabla\cdot\left(e^{\chi\;\Phi}\left(\nabla\;u\right)u^{p-1}\right)+\nu\;e^{\chi\;\Phi}\left(p-1\right)\;\dfrac{4}{p^2} \mid\nabla (u^{p/2})\mid^2+\alpha\;S\left(\Phi,T\right)e^{\chi\;\Phi}\;u^p\\
	\\
	+\left(\rho\;P\left(\Phi,T\right)-\left(\dfrac{p-1}{p}\right)\;\chi\;\gamma\;R\left(\Phi,T\right)\;\Phi\right)\;e^{\chi\;\Phi}\;u^p\left(\dfrac{e^{\chi\;\Phi}\;u+N+\Phi}{K}\right)\\
	\\
	=\left(\rho\;P\left(\Phi,T\right)-\left(\dfrac{p-1}{p}\right)\;\chi\;\gamma\;R\left(\Phi,T\right)\;\Phi\right)\;e^{\chi\;\Phi}\;u^p
	+\dfrac{\chi}{p}\;e^{\chi\;\Phi}\;u^p\;\delta\;Q\left(\Phi,T\right)\;\Phi.
	\end{array}
	\end{equation}
	
	Due to hypothesis $\left(\ref{rho_mayor_p_gamma}\right)$ and $\left(\ref{P_menor_R_Phi}\right)$, it is easy to see in $\left(\ref{L2H1_prueba1}\right)$ that, $$\rho\;P\left(\Phi,T\right)-\left(\dfrac{p-1}{p}\right)\;\chi\;\gamma\;R\left(\Phi,T\right)\;\Phi\geq0.$$ 
	
	Using now that $0\leq\Phi\leq K$, $\left(\ref{PQ_menor1}\right)$ and $\left(\ref{rho_mayor_p_gamma}\right)$ we obtain that
	
	\begin{equation}\label{L2Lp_prueba1}
	\begin{array}{c}
	\dfrac{1}{p}\left(e^{\chi\;\Phi}\;u^p\right)_t-\nu\;\nabla\cdot\left(e^{\chi\;\Phi}\left(\nabla\;u\right)u^{p-1}\right)+\nu\;e^{\chi\;\Phi}\left(p-1\right)\;\dfrac{4}{p^2} \mid\nabla (u^{p/2})\mid^2	+\alpha\;S\left(\Phi,T\right)e^{\chi\;\Phi}\;u^p\\
	\\
	\leq C\;e^{\chi\; \Phi} u^p
	\end{array}
	\end{equation}
	with $C>0$. Integrating $\left(\ref{L2Lp_prueba1}\right)$ in $\Omega$, it holds that
	
	\begin{equation}\label{L2Lp_prueba2}
	\begin{array}{c}
	\displaystyle\dfrac{1}{p}\;\dfrac{d}{dt}\int_{\Omega}e^{\chi\;\Phi}\;u^p\;dx+\nu\;\left(p-1\right)\;\dfrac{4}{p^2}\int_{\Omega}\;e^{\chi\;\Phi}\mid\nabla (u^{p/2})\mid^2\;dx+
	\alpha\int_{\Omega}S\left(\Phi,T\right)e^{\chi\;\Phi}\;u^p\;dx\\
	\\
	\leq\displaystyle C\int_{\Omega}\;e^{\chi\; \Phi} u^p\;dx
	\end{array}
	\end{equation}
	with $C>0$ independent of $p$ (along the proof, we will denote by $C$ different constants independent of $p$). 
	\\
%

Using the auxiliary variable $w=u^{p/2}$, we can rewrite $\left(\ref{L2Lp_prueba2}\right)$ as follows

 \begin{equation}\label{Linf_prueba1}
 \begin{array}{c}
 \displaystyle\dfrac{1}{p}\;\dfrac{d}{dt}\|e^{\frac{\chi\;\Phi}{2}}\;w\|_{L^2\left(\Omega\right)}^2+4\;\nu\;\dfrac{\left(p-1\right)}{p^2}\;\|e^{\frac{\chi\;\Phi}{2}}\;\nabla w\|_{L^2\left(\Omega\right)}^2\leq C\; \|e^{\frac{\chi\;\Phi}{2}}\;w\|_{L^2\left(\Omega\right)}^2.
 \end{array}
 \end{equation}

 Thus, applying Gronwall's lemma, we deduce for $p=2$ that
 
 $$\nabla u\;\;\text{is bounded in}\;\; L^2\left(0,T_f;L^2\left(\Omega\right)\right).$$
 
 Now, using the following equivalent norms with constants independent of $p$ 
 
  \begin{equation}\label{estimacion_exponencial}
 \displaystyle
  \| z \|_{L^2\left(\Omega\right)}^2 \le \| e^{\frac{\chi\;\Phi}{2}}\; z \|_{L^2\left(\Omega\right)}^2 \le e^{\chi\;K} \|  z \|_{L^2\left(\Omega\right)}^2, 
  \end{equation}
multiplying $\left(\ref{Linf_prueba1}\right)$ by $p$ and using that $\dfrac{p-1}{p}\geq \dfrac{1}{2}$ for any $p\geq2$, we obtain that
 
  \begin{equation}\label{Linf_prueba2}
 \begin{array}{c}
 \displaystyle\dfrac{d}{dt}\|e^{\frac{\chi\;\Phi}{2}}\;w\|_{L^2\left(\Omega\right)}^2+2\;\nu\;\|\nabla w\|_{L^2\left(\Omega\right)}^2\leq  C\;p\;\| w\|_{L^2\left(\Omega\right)}^2.
 \end{array}
 \end{equation}
 
 We are going to apply the following Gagliardo-Nirenberg interpolation inequality (\cite[Theorem 10.1]{Friedman_1969}) 
 
 \begin{equation}\label{Glagiardo}
 \| w\|_{L^2\left(\Omega\right)}^2\leq\varepsilon\|\nabla\;w\|_{L^2\left(\Omega\right)}^2+C\left(\dfrac{1}{\varepsilon}\right)^{n/2}\|w\|_{L^1\left(\Omega\right)}^2
 \end{equation}
 with $\varepsilon>0$ and $n$ the dimension of $\Omega$ (in this case $n=3$). Applying $\left(\ref{Glagiardo}\right)$ for
 $\varepsilon=\dfrac{\nu}{C\;p}$ in the right hand side of $\left(\ref{Linf_prueba2}\right)$, we deduce that
 
%

\begin{equation}\label{Linf_prueba4}
\begin{array}{c}
\displaystyle\dfrac{d}{dt}\|e^{\frac{\chi\;\Phi}{2}}\;w\|_{L^2\left(\Omega\right)}^2+\nu\;\|\nabla w\|_{L^2\left(\Omega\right)}^2\leq  C\;p^2\;\| w\|_{L^1\left(\Omega\right)}^2.
\end{array}
\end{equation}

Using $\left(\ref{Glagiardo}\right)$ in $\left(\ref{Linf_prueba4}\right)$ but now for $\varepsilon = \nu$, 
 it holds that

\begin{equation}\label{Linf_prueba5}
\begin{array}{c}
\displaystyle\dfrac{d}{dt}\|e^{\frac{\chi\;\Phi}{2}}\;w\|_{L^2\left(\Omega\right)}^2+\;\| w\|_{L^2\left(\Omega\right)}^2\leq  C\left(p^2+1\right)\| w\|_{L^1\left(\Omega\right)}^2.
\end{array}
\end{equation}

Finally, due to $\left(\ref{estimacion_exponencial}\right)$, we can deduce that 

\begin{equation}\label{Linf_prueba6}
\begin{array}{c}
\displaystyle\dfrac{d}{dt}\|e^{\frac{\chi\;\Phi}{2}}\;w\|_{L^2\left(\Omega\right)}^2+C_1\;\|e^{\frac{\chi\;\Phi}{2}}\; w\|_{L^2\left(\Omega\right)}^2\leq  C\left(p^2+1\right)\| w\|_{L^1\left(\Omega\right)}^2
\end{array}
\end{equation}
where $C_1 = e^{-\chi\;K}$.
\\

Hence, we obtain that
\begin{equation}\label{Linf_prueba7}
\begin{array}{c}
\displaystyle\max_{t\in (0,T_f)}\|u\|_{L^p\left(\Omega\right)}^p\leq\|e^{\frac{\chi\;\Phi}{2}}\;w\left(t\right)\|_{L^2\left(\Omega\right)}^2\leq \displaystyle e^{-C_1\;t}\;C\|u_0\|_{L^\infty\left(\Omega\right)}^p\\
\\
\displaystyle+C\left(p^2+1\right)\;e^{-C_1\;t}\int_{0}^{t}e^{C_1\;s}\left(\int_{\Omega}u^{p/2}\;dx\right)^2\;ds
\displaystyle\leq C\;\|u_0\|_{L^\infty\left(\Omega\right)}^p+C\left(p^2+1\right)\;\max_{t\in\left(0,T_f\right)}\|u\|_{L^{p/2}\left(\Omega\right)}^p \\
\\
\displaystyle\leq C\; \max\Big\{\left(p^2+1\right)\;\max_{t\in\left(0,T_f\right)}\|u\|_{L^{p/2}\left(\Omega\right)}^p,\;\|u_0\|_{L^\infty\left(\Omega\right)}^p\Big\}.
\end{array}
\end{equation}

Following a similar argument to used by Alikakos in \cite{Alikakos_1979} (see Appendix), from $\left(\ref{Linf_prueba7}\right)$ we can obtain that

$$u\text{ is bounded in } L^\infty\left(0,T_f;L^\infty\left(\Omega\right)\right).$$

As consequence, $T$ is bounded in $L^\infty\left(0,T_f;L^{\infty}\left(\Omega\right)\right)$. 
\\

Since $N_t=f_2\left(T,\Phi\right)$ and $T$ and $\Phi$ are bounded in $L^\infty\left(0,T_f;L^\infty\left(\Omega\right)\right)$ we obtain that $N$ is bounded in $L^\infty\left(0,T_f;L^{\infty}\left(\Omega\right)\right)$.
\end{proof}

\subsection{Proof of Theorem $\ref{teo_estimaciones}$ c)}
	
	Let $\left(T,N,\Phi\right)$ be a solution of $\left(\ref{problin}\right)$. Taking gradient in the second and third equation of $\left(\ref{problin}\right)$, 
	
	\begin{equation}\label{ec_grad_phi}
	\begin{array}{ll}
\left(\nabla\;\Phi\right)_t=&\gamma\left[\left(\dfrac{\partial\left(R\left(\Phi,T\right)\;\Phi\right)}{\partial\;\Phi}\;\nabla\;\Phi+\dfrac{\partial\left(R\left(\Phi,T\right)\;\Phi\right)}{\partial\;T}\;\nabla\;T\right)\left(1-\dfrac{T+N+\Phi}{K}\right)-\right.\\
\\
&\left.-\dfrac{R\left(\Phi,T\right)\;\Phi}{K}\;\left(\nabla\;T+\nabla\;N+\nabla\;\Phi\right)\right]-\delta\left(\dfrac{\partial\left(Q\left(\Phi,T\right)\;\Phi\right)}{\partial\;\Phi}\;\nabla\;\Phi+\right.\\
\\
&\left.+\dfrac{\partial\left(Q\left(\Phi,T\right)\;\Phi\right)}{\partial\;T}\;\nabla\;T\right),
	\end{array}
	\end{equation}



	\begin{equation}\label{ec_grad_n}
\begin{array}{ll}
\left(\nabla\;N\right)_t=&\alpha\left(\dfrac{\partial\left(S\left(\Phi,T\right)\;T\right)}{\partial\;\Phi}\;\nabla\;\Phi+\dfrac{\partial\left(S\left(\Phi,T\right)\;T\right)}{\partial\;T}\;\nabla\;T\right)+\\
\\
&+\delta\left(\dfrac{\partial\left(Q\left(\Phi,T\right)\;\Phi\right)}{\partial\;\Phi}\;\nabla\;\Phi+\dfrac{\partial\left(Q\left(\Phi,T\right)\;\Phi\right)}{\partial\;T}\;\nabla\;T\right).
\end{array}
\end{equation}
	
	Using the change of variable $T=e^{\chi\;\Phi}\;u$ as in Lemma $\ref{estimaciones3}$, we deduce that 
	\begin{equation}\label{nablaT}
	\nabla\;T=\chi\;e^{\chi\;\Phi}\;u\;\nabla\Phi+e^{\chi\;\Phi}\;\nabla u=\chi\;T\;\nabla\;\Phi+e^{\chi\;\Phi}\;\nabla u
		\end{equation}
		and we know from Lemma $\ref{estimaciones3}$ that $\nabla\;u$ is bounded in $L^2\left(0,T_f;L^2\left(\Omega\right)\right)$. 
%
	Taking into account that $T$ and $\Phi$ are bounded in $L^\infty\left(0,T_f;L^{\infty}\left(\Omega\right)\right)$, it holds that $$|\nabla\;T|\leq C\left(|\nabla\;\Phi|+|\nabla\;u|\right).$$ 
	
	Thus, rewriting $\left(\ref{ec_grad_phi}\right)$ and $\left(\ref{ec_grad_n}\right)$ in terms of $\nabla\;u$, 
	multiplying $\left(\ref{ec_grad_phi}\right)$ and $\left(\ref{ec_grad_n}\right)$ by $\nabla\;\Phi$ 
	 and $\nabla\;N$ 
	respectively and integrating in $\Omega$, 
we deduce 

\begin{equation}\label{ec_grad_Phi_p}
\begin{array}{ll}
\dfrac{1}{2}\;\dfrac{d}{dt}\|\nabla\;\Phi\|_{L^2\left(\Omega\right)}^2\leq&\displaystyle C_1\;\|\nabla\;\Phi\|_{L^2\left(\Omega\right)}^{2}+C_2\;\int_{\Omega}|\nabla\;u|\;|\nabla\;\Phi|\;dx
\displaystyle+C_3\;\int_{\Omega}|\nabla\;N|\;|\nabla\;\Phi|\;dx,
\end{array}
\end{equation}
and 
\begin{equation}\label{ec_grad_N_p}
\begin{array}{ll}
\dfrac{1}{2}\;\dfrac{d}{dt}\|\nabla\;N\|_{L^2\left(\Omega\right)}^2\leq&\displaystyle C_4\;\int_{\Omega}|\nabla\;\Phi|\;|\nabla\;N|\;dx
\displaystyle+C_5\;\int_{\Omega}|\nabla\;u|\;|\nabla\;N|\;dx,
\end{array}
\end{equation}
with $C_i>0$ for $i=1,\ldots,5$. In $\left(\ref{ec_grad_Phi_p}\right)$ and $\left(\ref{ec_grad_N_p}\right)$ we have applied
the inequality $$\int_{\Omega}v\;|\nabla\;u|\;|\nabla\;\Phi|\;dx\leq\|v\|_{L^\infty\left(\Omega\right)}\;\int_{\Omega}|\nabla\;u|\;|\nabla\;\Phi|\;dx$$ with $v=T,\;N,\;\Phi$ since $T$, $N$ and $\Phi$ are bounded in $L^\infty\left(0,T_f;\;L^{\infty}\left(\Omega\right)\right)$.
\\
 
 Using now Cauchy-Schwarz and Young's inequalities 
 in $\left(\ref{ec_grad_Phi_p}\right)$ and $\left(\ref{ec_grad_N_p}\right)$ 
%
%
%
and adding them, 
it holds that
\begin{equation}\label{suma_grad_n_phi}
\begin{array}{ll}
\dfrac{1}{2}\;\dfrac{d}{dt}\left(\|\nabla\;\Phi\|^2_{L^2\left(\Omega\right)}+\|\nabla\;N\|^2_{L^2\left(\Omega\right)}\right)\leq&\displaystyle\widehat{C}_1\left(\|\nabla\;\Phi\|_{L^{2}\left(\Omega\right)}^2+\|\nabla\;N\|_{L^{2}\left(\Omega\right)}^2\right)+\widehat{C}_2\;\|\nabla\;u\|_{L^{2}\left(\Omega\right)}^2,
\end{array}
\end{equation}
with $\widehat{C}_i>0$ for $i=1,2$. Since $\nabla u$ is bounded in $L^2\left(0,T_f;L^2\left(\Omega\right)\right)$, applying Gronwall's Lemma, it holds that
$$\nabla\;N\;\text{and}\;\nabla\;\Phi\;\text{are bounded in }\;L^\infty\left(0,T_f;L^{2}\left(\Omega\right)\right).$$

Finally, using $\left(\ref{nablaT}\right)$ in $\left(\ref{ec_grad_phi}\right)$ and $\left(\ref{ec_grad_n}\right)$, we obtain that
$$\left(\nabla\;N\right)_t\;\text{and}\;\left(\nabla\;\Phi\right)_t\;\text{are bounded in }\;L^2\left(0,T_f;L^{2}\left(\Omega\right)\right).$$

\begin{col}
	$\nabla T$ is bonded in $L^2\left(0,T_f;L^{2}\left(\Omega\right)\right)$.
\end{col}

%
%
%
%
%

\section{A FE numerical scheme}\label{esquema_espacio}

In this Section, we are going to design an uncoupled and linear fully discrete scheme to approach $\left(\ref{probOriginal}\right)$-$\left(\ref{condinicio}\right)$ by means of an Implicit-Explicit (IMEX) Finite Difference in time and $P_1$ continuous finite element with "mass-lumping" in space discretization. This scheme will preserve the pointwise estimates that appear in Lemma $\ref{estimaciones1}$  considering acute triangulations.
\\

Now we introduce the hypotheses required along this section.

\begin{enumerate}[a)]
	\item Let $0<T_f<+\infty$. We consider the uniform time partition $$\displaystyle\left(0,T_f\right]=\bigcup^{K_f-1}_{k=0}\left(t_{k},t_{k+1}\right],$$ with $t_k=k\:dt$ where $K_f\in\mathbb{N}$ and $dt=\dfrac{T_f}{K_f}$ is the time step. Let $\Omega\subseteq\mathbb{R}^2$ or $\mathbb{R}^3$ a bounded domain with polygonal or polyhedral lipschitz-continuous boundary.
	\item  Let $\left\{\mathcal{T}_h\right\}_{h>0}$ be a family of shape-regular, quasi-uniform triangulations of $\overline{\Omega}$ formed by acute N-simplexes (triangles in $2$D and tetrahedral in $3$D with all angles lowers than $\pi/2$), such that $$\overline{\Omega}=\displaystyle\bigcup_{\mathcal{K}\in\mathcal{T}_h}\mathcal{K},$$ where $h=\displaystyle\max_{\mathcal{K}\in\mathcal{T}_h}h_{\mathcal{K}}$, with $h_{\mathcal{K}}$ being the diameter of $\mathcal{K}$. We denote $\mathcal{N}_h=\left\{a_i\right\}_{i\in I}$ the set of all the nodes of $\mathcal{T}_h$. 
	\item Conforming piecewise linear, finite element spaces associated to $\mathcal{T}_h$ are assumed for approximating $H^1\left(\Omega\right)$:
	$$N_h=\left\{n_h\in\mathcal{C}^0\left(\overline{\Omega}\right)\;\;:\;\;n_h\vert_{\mathcal{K}}\in\mathcal{P}_1\left(\mathcal{K}\right),\;\;\forall\; \mathcal{K}\in\mathcal{T}_h\right\}$$
	and its Lagrange basis is denoted by $\left\{\varphi_a\right\}_{a\in\mathcal{N}_h}$.
\end{enumerate}

Let $I_h:\mathcal{C}^0\left(\overline{\Omega}\right)\rightarrow N_h$ be the nodal interpolation operator and consider the discrete inner product
$$\left(n_h,\overline{n}_h\right)_h=\int_{\Omega}I_h\left(n_h\cdot\overline{n}_h\right)=\sum_{a\in \mathcal{N}_h}n_h\left(a\right)\;\overline{n}_h\left(a\right)\int_{\Omega}\varphi_a,\quad\forall n_h,\overline{n}_h\in N_h$$
which induces the discrete norm $\| n_h\|_h=\sqrt{\left(n_h,n_h\right)_h}$ defined on $N_h$ (that is equivalent to $L^2\left(\Omega\right)$-norm).
\\

Before building the numerical scheme, we will transform the first equation of $\left(\ref{probOriginal}\right)$ into a non-linear diffusion equation throughout the change of variable $T=u\;e^{\chi\;\Phi}$ as in Lemma $\ref{estimaciones3}$. Therefore, the first equation of $\left(\ref{probOriginal}\right)$ changes to:

\begin{equation}\label{ecuacion_u}
e^{\chi\;\Phi}\;u_t-\nu\nabla\cdot\left(e^{\chi\;\Phi}\;\nabla\;u\right)=\widehat{f}_1\left(u,N,\Phi\right)
\end{equation}
where 
\begin{equation}\label{f1gorro}
\begin{array}{ll}
\widehat{f}_1\left(u,N,\Phi\right)=&T\left[\rho\;P\left(\Phi,T\right)+\chi\;\Phi\left(\gamma\;R\left(\Phi,T\right)\left(\dfrac{T+N+\Phi}{K}\right)+\delta\;Q\left(\Phi,T\right)\right)\right]\\
\\
&-T\left[\rho\;P\left(\Phi,T\right)\left(\dfrac{T+N+\Phi}{K}\right)+\alpha\;S\left(\Phi,T\right)+\chi\;\gamma\;R\left(\Phi,T\right)\;\Phi\right].
\end{array}
\end{equation}

Thus, 
we consider the following linear uncoupled numerical scheme for $\left(\ref{ecuacion_u}\right)$ jointly with $\left(\ref{probOriginal}\right)_b$ and $\left(\ref{probOriginal}\right)_c$: given $u^k_h, N^k_h, \Phi^k_h\in N_h$, find $u_{h}^{k+1}, N_{h}^{k+1}, \Phi_{h}^{k+1}\in N_h$ in a decoupled way (first $\Phi$, then $u$ and finally $N$) satisfying
\begin{equation}\label{eqT_space}
\begin{array}{ccl}
\left(e^{\chi\;\Phi^k_h}\;\delta_t u_{h}^{k+1}, v\right)_h+\nu\;\left(e^{\chi\;\Phi^k_h}\;\nabla\;u_{h}^{k+1},\nabla v\right)&=&\left(\left(\widehat{f}_1\right)_h^k,v\right)_h,\quad \forall v\in N_h,
\end{array}
\end{equation}
\begin{equation}\label{eqN_space}
\begin{array}{ccl}
\delta_t N_{h}^{k+1}\left(a\right) &=&\left(\widehat{f}_2\right)_{h}^{k}\left(a\right),\quad \forall a\in\mathcal{N}_h,\\
\end{array}
\end{equation}
\begin{equation}\label{eqF_space}
\begin{array}{ccl}
\delta_t \Phi_{h}^{k+1}\left(a\right) &=&\left(\widehat{f}_3\right)_{h}^{k}\left(a\right),\quad \forall a\in\mathcal{N}_h.\\
\end{array}
\end{equation}

We have denoted
	$$\delta_t u_{h}^{k+1}=\dfrac{u_h^{k+1}-u_h^{k}}{dt}$$ 
	and similarly for $\delta_t N_h^{k+1}$ and $\delta_t \Phi_h^{k+1}$. The approximation of the initial conditions are taken as
\begin{equation}\label{condinicio_espacio}
u^0_{h}=I_h\left(u_0\right)\in N_h,\;\;N^0_{h}=I_h\left(N_0\right)\in N_h ,\;\;\Phi^0_{h}=I_h\left(\Phi_0\right)\in N_h
\end{equation}
where we consider for simplicity that $T_0,\;N_0,\Phi_0\in\mathcal{C}^0\left(\overline{\Omega}\right)$ with $u_0=e^{-\chi\;\Phi_0}\;T_0$. 
\\

Finally, the functions $\left(\widehat{f}_i\right)_{h}^{k}$ for $i=1,2,3$ in $\left(\ref{eqT_space}\right)$, $\left(\ref{eqN_space}\right)$ and $\left(\ref{eqF_space}\right)$, have the following definitions:

\begin{equation}\label{f1_space}
\begin{array}{ll}
\left(\widehat{f}_1\right)_{h}^{k}=&T^k_h\left(\rho\;P^k_h+\chi\;\Phi^{k}_h\left(\gamma\;R^k_h\left(\dfrac{T_{h}^{k}+N^k_h+\Phi^k_h}{K}\right)+\delta\;Q^k_h\right)\right)-\\
\\
&\displaystyle- T^{k+1}_h\left(\rho\;P^k_h\left(\dfrac{T_{h}^{k}+N^k_h+\Phi^k_h}{K}\right)+\alpha\;S^k_h+\chi\;\gamma\;R^k_h\;\Phi^{k}_h\right),
\end{array}
\end{equation}

\begin{equation}\label{f2_space}
\begin{array}{ll}
\left(\widehat{f}_2\right)_{h}^{k}=&\displaystyle \alpha\;S^k_h\;T_{h}^{k+1}+ \delta\;Q^k_h\;\Phi^{k+1}_h,
\end{array}
\end{equation}	

\begin{equation}\label{f3_space}
\begin{array}{ll}
\left(\widehat{f}_3\right)_{h}^{k}=&\gamma\;R^k_h\;\Phi^{k}_h\left(1-\dfrac{\Phi^{k+1}_h}{K}\right)-\Phi^{k+1}_h\left(\gamma\;R^k_h\;\dfrac{T_{h}^{k}+N^k_h}{K}+  \delta\;Q^k_h\;\Phi^{k+1}_h\right). 
\end{array}
\end{equation}

The functions $P^k_h$,  $S^k_h$, $R^k_h$ and $Q^k_h$ in $\left(\ref{f1_space}\right)$-$\left(\ref{f3_space}\right)$, are the corresponding dimensionless factors $P\left(\Phi^k_h,T^k_h\right)$, $S\left(\Phi^k_h,T^k_h\right)$, $R\left(\Phi^k_h,T^k_h\right)$ and $Q\left(\Phi^k_h,T^k_h\right)$ defined in $\left(\ref{funcion_P}\right)$-$\left(\ref{funcion_Q}\right)$ with $T_h^k=e^{\chi\;\Phi^{k}_h}\;u_{h}^{k}$ and $T^{k+1}_h=e^{\chi\;\Phi^{k+1}_h}\;u^{k+1}_h$.

%
%

\begin{obs}
	There exists an unique solution of scheme $\left(\ref{eqT_space}\right)$-$\left(\ref{f3_space}\right)$  because:
	\begin{enumerate}
		\item $\Phi_{h}^{k+1}\left(a\right)$ can be computed directly from $\left(\ref{eqF_space}\right)$.
		\item There exists an unique $u_{h}^{k+1}$ solution of $\left(\ref{eqT_space}\right)$ by Lax-Milgram theorem.
		\item $N_{h}^{k+1}\left(a\right)$ can be computed directly from $\left(\ref{eqN_space}\right)$.
	\end{enumerate}
\end{obs}
\subsection{Proof of Theorem $\ref{teo_estimaciones_discreto}$}

In this part, we are going to get a priori energy estimates for the fully discrete solution $u_h^{k+1}$, $N_h^{k+1}$ and $\Phi_{h}^{k+1}$ (and hence, for $T_h^{k+1}$) of $\left(\ref{eqT_space}\right)$, $\left(\ref{eqN_space}\right)$ and $\left(\ref{eqF_space}\right)$ which are independent of $(h, k)$. 
\\

The following result is based on the hypothesis of acute triangulations to get a discrete maximum principle, see  \cite{Ciarlet_1973}. In fact, we arrive at discrete version of Lemma $\ref{estimaciones1}$.

\begin{lem}[Proof of Theorem $\ref{teo_estimaciones_discreto}$]\label{positivo_space}
	Let $u_h^k,\;N_h^k,\;\Phi_{h}^k\in N_h$ with $T_h^k=e^{\chi\;\Phi^{k}_h}\;u_{h}^{k}$ such that $0\leq u_h^k,\;N_h^k,\;\Phi_{h}^k$ in $\Omega$ (in particular $T_h^k\geq0$ in $\Omega$). Then, $0\leq\Phi_{h}^{k+1}\leq K$ and $u_h^{k+1},\;N_h^{k+1}\geq0$ in $\Omega$ (and also $T_h^{k+1}\geq0$).
\end{lem}
\begin{proof}
	\begin{itemize}
		\item Step $1$. $\Phi_{h}^{k+1}\geq0$.
		\\
		
		Multiplying $\left(\ref{eqF_space}\right)$ by $(\Phi_h^{k+1}(a))_-$ and using that $\Phi_h^k(a)\ge 0$, it holds that:
		\begin{equation}\label{positividad_F_numerico}
		\dfrac{1}{dt}\;\left(\Phi_h^{k+1}\left(a\right)\right)_-^2 \leq
		\left(\widehat{f}_3\right)_h^k\left(a\right)\;\left(\Phi_h^{k+1}\left(a\right)\right)_-.
		\end{equation}
		
		Indeed, using the form of $\left(\widehat{f}_3\right)_h^k$ given in $\left(\ref{f3_space}\right)$, the following estimates hold
		
		$$\gamma\;R^k_h\left(a\right)\;\left(\Phi^{k}_h\left(a\right)\right)\left(\Phi_h^{k+1}\left(a\right)\right)_-\leq0$$
		and
		$$- \left(\gamma\;R^k_h\left(a\right)\;\left(\dfrac{T_h^k\left(a\right)+N^{k}_h\left(a\right)+\Phi_{h}^k\left(a\right)}{K}\right)+\delta\; Q^{k}_h\left(a\right)\right)\left(\Phi_h^{k+1}\left(a\right)\right)\left(\Phi_h^{k+1}\left(a\right)\right)_-\leq 0.$$
		
		Adding the last two inequalities, one has
		\begin{equation}\label{positividad_F_numerico2}
		\left(\widehat{f}_3\right)_h^k\left(a\right)\;\left(\Phi_h^{k+1}\left(a\right)\right)_-\leq0.
		\end{equation}
		
		Therefore, from $\left(\ref{positividad_F_numerico}\right)$ and $\left(\ref{positividad_F_numerico2}\right)$, $\left(\Phi_h^{k+1}\left(a\right)\right)_-\equiv0$ $\forall a\in\mathcal{N}_h$ and this implies $\Phi_h^{k+1}\geq0$ in $\Omega$.
		
		\item Step $2$. $\Phi_h^{k+1}\leq K$.
		\\
		
		Multiplying $\left(\ref{eqF_space}\right)$ by $\left(\left(\Phi_h^{k+1}-K\right)\left(a\right)\right)_+$, it holds that
		
	   \begin{equation}\label{cota_superior_F_numerico}
		\dfrac{1}{dt}\;\left(\left(\Phi_h^{k+1}-K\right)\left(a\right)\right)_+^2 \leq\left(\widehat{f}_3\right)_{h}^{k}\left(a\right)\;\left(\Phi_h^{k+1}\left(a\right)-K\right)_+
		\end{equation}
		
		On the other hand, since in every node $a\in\mathcal{N}$, due to the form of $\left(\widehat{f}_3\right)_{h}^{k}$ given in $\left(\ref{f3_space}\right)$ the following estimates hold
		
		$$\left(\gamma\;R^k_h\left(a\right)\;\Phi^{k}_h\left(a\right)\left(1-\dfrac{\Phi^{k+1}_h\left(a\right)}{K}\right)\right)\left(\Phi_h^{k+1}\left(a\right)-K\right)_+\leq0$$
		and
		$$- \left(\gamma\;R^k_h\left(a\right)\;\left(\dfrac{T_h^k\left(a\right)+N^{k}_h\left(a\right)}{K}\right)+\delta\; Q^{k}_h\left(a\right)\right)\left(\Phi_h^{k+1}\left(a\right)\right)\left(\Phi_h^{k+1}\left(a\right)-K\right)_+\leq 0.$$
		
		Thus, adding the last two inequalities, we obtain that
		\begin{equation}\label{cota_superior_F_numerico2}
		\left(\widehat{f}_3\right)_{h}^{k}\left(a\right)\;\left(\Phi_h^{k+1}\left(a\right)-K\right)_+\leq0.
		\end{equation}
		
		Therefore, from $\left(\ref{cota_superior_F_numerico}\right)$ and $\left(\ref{cota_superior_F_numerico2}\right)$, $\left(\Phi_h^{k+1}\left(a\right)-K\right)_+\equiv0$ $\forall a\in\mathcal{N}_h$ and this implies $\Phi_h^{k+1}\leq K$ in $\Omega$.
		
\item Step $3$. $u_h^{k+1}\geq0$.
	\\
	
	Let $I_h((u_h^{k+1})_-)\in N_h$ be defined as 
	$$I_h\left(\left(u_h^{k+1}\right)_-\right)=\sum_{a\in \mathcal{N}_h}\left(u_h^{k+1}\left(a\right)\right)_-\varphi_a,$$
	where $\left(u_h^{k+1}\left(a\right)\right)_-=\min\left\{0,u_h^{k+1}\left(a\right)\right\}$. Analogously, one defines $I_h((u_h^{k+1})_+)\in N_h$ as
	$$I_h\left(\left(u_h^{k+1}\right)_+\right)=\sum_{a\in \mathcal{N}_h}\left(u_h^{k+1}\left(a\right)\right)_+\varphi_a,$$
	where $\left(u_h^{k+1}\left(a\right)\right)_+=\max\left\{0,u_h^{k+1}\left(a\right)\right\}$. Notice that $u_h^{k+1}=I_h((u_h^{k+1})_-)+I_h((u_h^{k+1})_+)$.
	\\
	
	Choosing $v=I_h((u_h^{k+1}(a))_-)$ in $\left(\ref{eqT_space}\right)$, it follows that,
	
	\begin{equation}\label{T_space_positiva}
	\begin{array}{c}
	\dfrac{1}{dt}\Big\|\left(u_h^{k+1}\right)_-\Big\|_h^2 +\nu\;\left(\left(e^{\chi\;\Phi^k_h}\right)\nabla u_h^{k+1},\nabla I_h\left(\left(u_h^{k+1}\right)_-\right)\right)\leq\\
	\\
	\leq\left(\widehat{f}_1\left(u_{h}^{k},u_{h}^{k+1},N_{h}^{k},\Phi_{h}^{k}\right),\left(u_h^{k+1}\right)_-\right)_h,
	\end{array}
	\end{equation}
	where we have used in the left hand side that $$\dfrac{1}{dt}\Big\|\left(u_h^{k+1}\right)_-\Big\|_h^2 \leq	\dfrac{1}{dt}\Big\|\left(e^{\frac{\chi\;\Phi^k_h}{2}}\right)\;\left(u_h^{k+1}\right)_-\Big\|_h^2$$ and that in every node $a\in \mathcal{N}_h$, 
	$$	
	\begin{array}{ll}
	\delta_tu_{h}^{k+1}\left(a\right)\cdot\left(u_h^{k+1}\left(a\right)\right)_-=&	\dfrac{1}{dt}\left(\Big|\left(u_{h}^{k+1}\left(a\right)\right)_-\Big|^2-u_{h}^{k}\left(a\right)\cdot\left(u_h^{k+1}\left(a\right)\right)_-\right)\geq\\
	\\	
	&\geq \dfrac{1}{dt}\left(\Big|\left(u_{h}^{k+1}\left(a\right)\right)_-\Big|^2\right)
	\end{array}
	$$ 
	using that $e^{\chi\;\Phi^k_h\left(a\right)}>0$, $u_{h}^{k}\left(a\right)\geq0$ and $\left(u_h^{k+1}\left(a\right)\right)_-\leq0$. On the other hand, we can make the following decomposition in the diffusion term
	$$
	\left(\left(e^{\chi\;\Phi^k_h}\right)\nabla u_h^{k+1},\nabla I_h\left(\left(u_h^{k+1}\right)_-\right)\right) =\left(\left(e^{\chi\;\Phi^k_h}\right)\nabla I_h\left(\left(u_h^{k+1}\right)_-\right),\nabla I_h\left(\left(u_h^{k+1}\right)_-\right)\right)+$$
	$$+\left(\left(e^{\chi\;\Phi^k_h}\right)\nabla I_h\left(\left(u_h^{k+1}\right)_+\right),\nabla I_h\left(\left(u_h^{k+1}\right)_-\right)\right)
	=\Big\|\left(e^{\chi\;\Phi^k_h}\right)^{1/2}\nabla I_h\left(\left(u_h^{k+1}\right)_-\right)\Big\|_{L^2\left(\Omega\right)}^2+$$
	$$\displaystyle+\sum_{a\neq\widetilde{a}\in \mathcal{N}_h} \left(u_h^{k+1}\left(a\right)\right)_- \left(u_h^{k+1}\left(\widetilde{a}\right)\right)_+\left(\left(e^{\chi\;\Phi^k_h}\right)\nabla\varphi_a,\nabla\varphi_{\widetilde{a}}\right).
	$$
	
	Hence, using that $\left(u_h^{k+1}\left(a\right)\right)_- \;\left(u_h^{k+1}\left(\widetilde{a}\right)\right)_+\leq0$ if $a\neq\widetilde{a}$, $e^{\chi\;\Phi^k_h\left(a\right)}$ is a positive function and that
	$$\nabla\varphi_a\cdot\nabla\varphi_{\widetilde{a}}\leq0\quad\forall a\neq\widetilde{a}\in \mathcal{N}_h$$ (owing to the hypothesis of acute triangulation), we deduce,
	\begin{equation}\label{T_diagonal}
	\begin{array}{c}
	\left(\left(e^{\chi\;\Phi^k_h}\right)\nabla u_h^{k+1},\nabla I_h\left(\left(u_h^{k+1}\right)_-\right)\right)\geq\Big\|\left(e^{\chi\;\Phi^k_h}\right)^{1/2}\nabla I_h\left(\left(u_h^{k+1}\right)_-\right)\Big\|_{L^2\left(\Omega\right)}^2.
	\end{array}
	\end{equation}
	
	Adding $\left(\ref{T_diagonal}\right)$ in $\left(\ref{T_space_positiva}\right)$, it holds that
	
	\begin{equation}\label{positividad_T_numerico}
	\begin{array}{c}
	\dfrac{1}{dt}\Big\|\left(u_h^{k+1}\right)_-\Big\|_h^2 +\nu\;\Big\|\left(e^{\chi\;\Phi^k_h}\right)^{1/2}\nabla I_h\left(\left(u_h^{k+1}\right)_-\right)\Big\|_{L^2\left(\Omega\right)}^2
	\leq\left(\left(\widehat{f}_1\right)_{h}^{k},\left(u_h^{k+1}\right)_-\right)_h.
	\end{array}
	\end{equation}
	
	On the other hand, by using that in every node $a\in \mathcal{N}_h$, due to the form of $\left(\widehat{f}_1\right)_{h}^{k}$ given in $\left(\ref{f1_space}\right)$, the following estimates hold
	$$\left(\rho\;P^k_h\left(a\right)+\chi\;\Phi^{k}_h\left(a\right)\left(\gamma\;R^k_h\left(a\right)\left(\dfrac{T_{h}^{k}\left(a\right)+N^k_h\left(a\right)+\Phi^k_h\left(a\right)}{K}\right).+\delta\;Q^k_h\left(a\right)\right)\right)\left(T_h^{k}\left(a\right)\right)\left(u_h^{k+1}\left(a\right)\right)_-\leq0$$
	and
	$$- \left(\rho\;P^k_h\left(a\right)\left(\dfrac{T_{h}^{k}\left(a\right)+N^k_h\left(a\right)+\Phi^k_h\left(a\right)}{K}\right)+\alpha\;S^k_h\left(a\right)+\chi\;\gamma\;R^k_h\left(a\right)\;\Phi^{k}_h\left(a\right)\right)\left(T_h^{k+1}\left(a\right)\right)\left(u_h^{k+1}\left(a\right)\right)_-\leq 0$$
	owing to $\left(T_h^{k+1}\left(a\right)\right)\left(u_h^{k+1}\left(a\right)\right)_-=\left(\left(u_h^{k+1}\left(a\right)\right)_-\right)^2\;e^{\chi\;\Phi^{k+1}_h}\left(a\right)\geq0$.
	\\
	
	Then, adding the last two inequalities, we obtain that
	\begin{equation}\label{positividad_T_numerico2}
	\left(\left(\widehat{f}_1\right)_{h}^{k},\left(u_h^{k+1}\right)_-\right)_h\leq0.
	\end{equation}
	
	Therefore, from $\left(\ref{positividad_T_numerico}\right)$ and $\left(\ref{positividad_T_numerico2}\right)$, $\left(u_h^{k+1}\right)_-\equiv0$ and this implies $u_h^{k+1}\geq0$ in $\Omega$. As we recover $T^{k+1}_h$ from $u^{k+1}_h$ and $\Phi^{k+1}_h$ as $T^{k+1}_h=e^{\chi\;\Phi^{k+1}_h}\;u^{k+1}_h$, we have in particular that $T^{k+1}_h\geq0$ in $\Omega$.

	\item Step $4$. $N_h^{k+1}\geq0$.
	\\
	
	Finally, for $\left(\ref{eqN_space}\right)$  it is easy to obtain that 
	\begin{equation}\label{positividad_N_numerico}
	\dfrac{1}{dt}\;\left(N_h^{k+1}\left(a\right)\right)_-^2 \leq\left(\widehat{f}_2\right)_{h}^{k}\left(a\right)\;\left(N_h^{k+1}\left(a\right)\right)_-.
	\end{equation}
	
	In addition, $\left(\widehat{f}_2\right)_{h}^{k}\left(a\right)\geq0$ in every node $a\in\mathcal{N}_h$ due to the form of $\left(\widehat{f}_2\right)_{h}^{k}$ given in $\left(\ref{f2_space}\right)$. Hence,
	\begin{equation}\label{positividad_N_numerico2}
	\left(\widehat{f}_2\right)_{h}^{k}\left(a\right)\;\left(N_h^{k+1}\left(a\right)\right)_-\leq0.
	\end{equation}
	
	Thus, from $\left(\ref{positividad_N_numerico}\right)$ and $\left(\ref{positividad_N_numerico2}\right)$, $\left(N_h^{k+1}\left(a\right)\right)_-\equiv0$ $\forall a\in\mathcal{N}_h$ and this implies $N_h^{k+1}\geq0$ in $\Omega$.
		\end{itemize}
\end{proof}

\section{Adimensionalization}\label{adimensionalizacion}

Here, we simplify the number of the parameters of $\left(\ref{probOriginal}\right)$ and present the simulations according to the dimensionless parameters. For that, we consider as one possible example of the dimensionless factors $P\left(\Phi,T\right)$, $S\left(\Phi,T\right)$, $Q\left(\Phi,T\right)$ and $R\left(\Phi,T\right)$ appearing in $\left(\ref{funciones}\right)$ satisfying the hypotheses $\left(\ref{P_menor_R_Phi}\right)$-$\left(\ref{parciales_S_Q_T_L_infitnito}\right)$, the following ones:

\begin{equation}\label{funcion_P}
P\left(\Phi,T\right)=\dfrac{\Phi}{\Phi+T},
\end{equation}

\begin{equation}\label{funcion_S}
S\left(\Phi,T\right)=\dfrac{K-\Phi}{T+\Phi+K},
\end{equation}

\begin{equation}\label{funcion_R}
R\left(\Phi,T\right)=\dfrac{T}{\dfrac{T^2}{K}+\Phi+K},
\end{equation}
and
\begin{equation}\label{funcion_Q}
Q\left(\Phi,T\right)=\dfrac{T}{\Phi+T}.
\end{equation}

These factors $P$, $S$, $Q$ and $R$ satisfy the conditions $\left(\ref{PQ_menor1}\right)$-$\left(\ref{q}\right)$. 
%
To verify $\left(\ref{P_menor_R_Phi}\right)$, observe that
$$
\frac{R(\Phi,T)\Phi}{P(\Phi,T)}=\frac{T(\Phi+T)}{T^2/K+\Phi+K}\leq C_1
$$
for some $C_1>0$. 
\\

Moreover, in $\left(\ref{funcion_S}\right)$ and $\left(\ref{funcion_R}\right)$ we consider a regularization in the denominator since without this regularization, the partial derivatives of $\left(\ref{funcion_S}\right)$ and $\left(\ref{funcion_R}\right)$ degenerate in $\left(0,0\right)$. 
\\

For the adimensionalization, we start studying the carrying capacity parameter $K>0$. We consider the change of variables 
$\widetilde{T}=\dfrac{T}{K}$, $\widetilde{N}=\dfrac{N}{K}$ and $\widetilde{\Phi}=\dfrac{\Phi}{K}$ passing the normalized capacity equal to $1$.
\\

Now, we consider the diffusion parameter $\nu$ and the tumor proliferation parameter $\rho$. We know that $\rho$ is related to the time variable while $\nu$ is related to the spatial variable. Thus, we can make the following change of the independent variables:
\begin{equation}\label{cambios_variable}
\left\{
\begin{array}{lcl}
s=\rho\;t&\Rightarrow& ds=\rho\;dt,\\
\\
y=\sqrt{\dfrac{\rho}{\nu}}\;x&\Rightarrow& dy=\sqrt{\dfrac{\rho}{\nu}}\;dx.
\end{array}
\right.
\end{equation}

Applying these changes in $\left(\ref{probOriginal}\right)$, it holds that 

\begin{equation}\label{prob_K_rho}
\left\{\begin{array}{ccl}
\dfrac{\partial \widetilde{T}}{\partial s}-\Delta\widetilde{T}+K\;\dfrac{\kappa}{\nu}\;\nabla\cdot\left(\widetilde{T}\;\nabla\widetilde{\Phi}\right)&=&\widetilde{f_1}\left(\widetilde{T},\widetilde{N},\widetilde{\Phi}\right)\\
&&\\
\dfrac{\partial \widetilde{N}}{\partial s}& = & \widetilde{f_2}\left(\widetilde{u},\widetilde{\Phi}\right)\\
&&\\
\dfrac{\partial \widetilde{\Phi}}{\partial s} & = &\widetilde{f_3}\left(\widetilde{u},\widetilde{N},\widetilde{\Phi}\right)\\
\end{array}\right.\end{equation}
where
\begin{equation}\label{funciones_K_rho}
\left\{\begin{array}{lll}
\widetilde{f_1}\left(\widetilde{T},\widetilde{N},\widetilde{\Phi}\right) &=&P\left(\widetilde{\Phi},\widetilde{T}\right)\;\widetilde{T}\left(1-\left(\widetilde{T}+\widetilde{N}+\widetilde{\Phi}\right)\right)-\dfrac{\alpha}{\rho}\;S\left(\widetilde{\Phi},\widetilde{T}\right)\;\widetilde{T},\\
\\
\widetilde{f_2}\left(\widetilde{T},\widetilde{\Phi}\right) &=& \dfrac{\alpha}{\rho}\;S\left(\widetilde{\Phi},\widetilde{T}\right)\;\widetilde{T}+\dfrac{\delta}{\rho}\;Q\left(\widetilde{\Phi},\widetilde{T}\right)\;\widetilde{\Phi},\\
\\
\widetilde{f_3}\left(\widetilde{T},\widetilde{N},\widetilde{\Phi}\right) &=& \dfrac{\gamma}{\rho}\;R\left(\widetilde{\Phi},\widetilde{T}\right)\;\widetilde{\Phi}\left(1-\left(\widetilde{T}+\widetilde{N}+\widetilde{\Phi}\right)\right)-\dfrac{\delta}{\rho}\;Q\left(\widetilde{\Phi},\widetilde{T}\right)\;\widetilde{\Phi}.
\end{array}\right.
\end{equation}

Hence, we obtain the following dimensionless parameters:

\begin{table}[H]
	\centering
	\begin{tabular}{c|c|c|c|c}
		\textbf{Dimensionless parameter} &$\kappa^*$ & $\alpha^*$   &	$\gamma^*$&	$\delta^*$ \\
		\hline
		\textbf{Original parameter} &$\vspace{2cm}K\;\dfrac{\kappa}{\nu}$&$\dfrac{\alpha}{\rho}$  &  $\dfrac{\gamma}{\rho}$ & $\dfrac{\delta}{\rho}$  
		\vspace{-2cm}
	\end{tabular}
	\caption{\label{parametros_K_rho}  Dimensionless parameters.}
\end{table}

Thus, we reduced three parameter from the original model $\left(\ref{probOriginal}\right)$: $\rho$, $\nu$ and $K$. 

\begin{obs}
	To simplify the notation, we consider along the rest of the paper: $s=t$, $y=x$, $\kappa^*=\kappa$, $\alpha^*=\alpha$, $\gamma^*=\gamma$, $\delta^*=\delta$, $\widetilde{T}=T$, $\widetilde{N}=N$, $\widetilde{\Phi}=\Phi$ and $\widetilde{f}_i=f_i$ for $i=1,2,3$.
\end{obs}

Finally, the adimensionalizated system is the following:

\begin{equation}\label{prob_K_rho2}
\left\{\begin{array}{ccl}
\dfrac{\partial T}{\partial t}-\Delta T+\kappa\;\nabla\cdot\left(T\;\nabla\Phi\right)&=&P\left(\Phi,T\right)\;T\left(1-\left(T+N+\Phi\right)\right)-\alpha\;S\left(\Phi,T\right)\;T,\\
&&\\
\dfrac{\partial N}{\partial t}& = &\alpha\;S\left(\Phi,T\right)\;T+\delta\;Q\left(\Phi,T\right)\;\Phi,\\
&&\\
\dfrac{\partial \Phi}{\partial t} & = &\gamma\;R\left(\Phi,T\right)\;\widetilde{\Phi}\left(1-\left(T+N+\Phi\right)\right)-\delta\;Q\left(\Phi,T\right)\;\Phi.
\end{array}\right.\end{equation}

\section{Numerical Simulations}\label{numerical_simulations}

In this section, we will show some numerical simulations in order to detect which parameters of $\left(\ref{prob_K_rho2}\right)$ are more important in the behaviour of the ring width between necrosis and tumor and the regular or irregular growth of the surface of a GBM.
\\

For the numerical simulations we will use the uncoupled and linear fully discrete scheme defined in $\left(\ref{eqT_space}\right)$-$\left(\ref{eqF_space}\right)$ by means of an Implicit-Explicit (IMEX) Finite Difference in time approximation and $P_1$ continuous finite element with "mass-lumping" in space. 
\\

We will use the computational domain, $\Omega=\left(-9,9\right)\times\left(-9,9\right)$, the final time, $T_f=500$, the structured triangulation, $\left\{\mathcal{T}_h\right\}_{h>0}$ of $\overline{\Omega}$ such that $\overline{\Omega}=\displaystyle\bigcup_{\mathcal{K}\in\mathcal{T}_h}\mathcal{K}$, partitioning the edges of $\partial\Omega$ into $45$ subintervals, corresponding with the mesh size $h=0.4$ and the time step, $dt=10^{-3}$.
\\

We consider along the work necrosis zero initially and initial tumor given by:

\begin{figure}[H]
		\includegraphics[width=8cm, height=5cm]{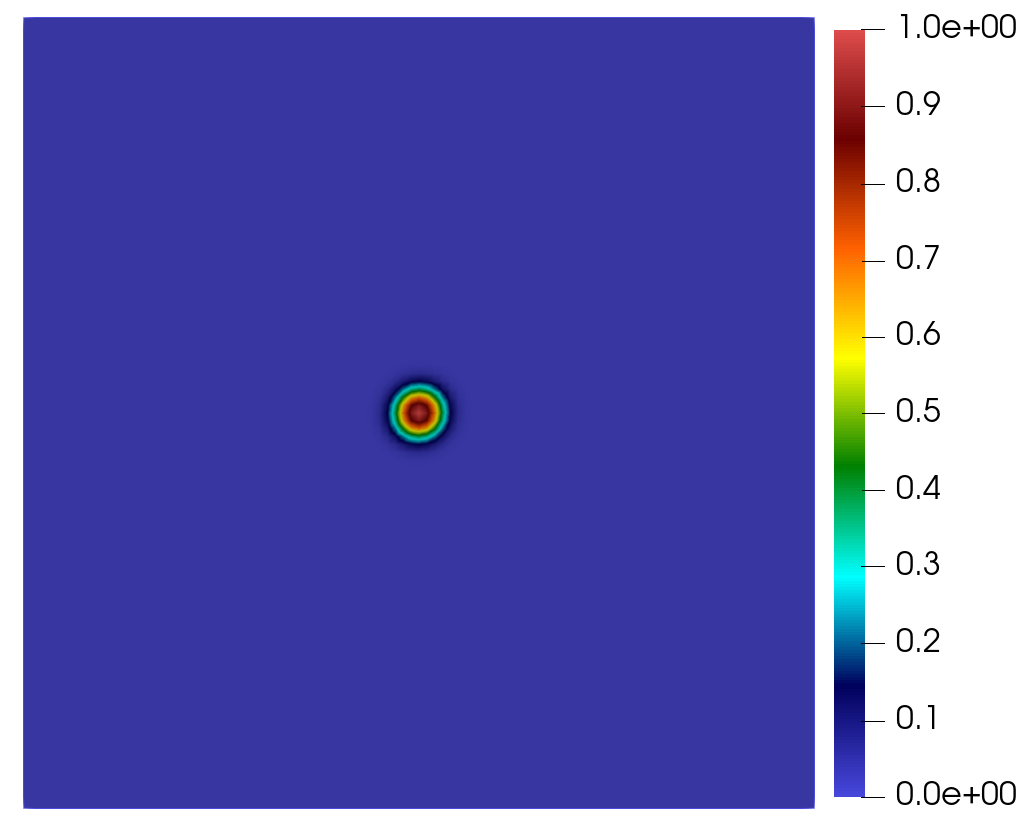}
		\centering
		\caption{Initial tumor.}
		\label{tumor_incial}
\end{figure}

For the vasculature, we will take different initial conditions depending on the kind of tumor growth considered.


\subsection{Ring width}\label{anillo}
Here, we present some numerical simulations according to the tumor-ring. Based on the study \cite{Julian_2016}, we know that tumors with a thick tumor ring have the worst prognosis as we can see in the Figure $\ref{curva_ring}$.

\begin{figure}[H]
\includegraphics[width=10cm, height=6cm]{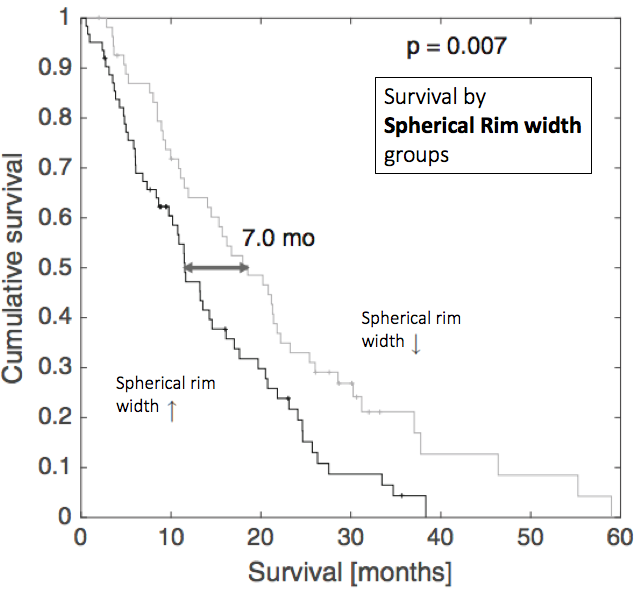}
\centering
\caption{Survival vs the ring width of GBM.}
\label{curva_ring}
\end{figure}

In order to measure different rings, we will compare the density of tumor with respect to the density of tumor and necrosis. In every simulation, we will change the value of one parameter and testing how the tumor growth changes. 
\\

Since the subjects of study are tumor and necrosis, we change the parameters of the tumor and necrosis equations, these are, $\kappa$ and $\alpha$. Then, in all the simulations the value of $\gamma$ and $\delta$ are fixed (see Table $\ref{parametros_fijos_razon_tumor_necrosis}$).

\begin{table}[H]
	\centering
	\begin{tabular}{c|c|c}
		\textbf{Variable} &$\gamma$ & 	$\delta$  \\
		\hline
		\textbf{Value} &   $0.255$  &$2.55$     \\
	\end{tabular}
	\caption{\label{parametros_fijos_razon_tumor_necrosis} Fixed value parameters.}
\end{table}

For $\kappa$ and $\alpha$, we will take either $\kappa=5$ and $\alpha\in\left[10,\;100\right]$ or $\kappa\in\left[1,\;10\right]$ and $\alpha=45$ (see Table $\ref{parametros_variables_razon_tumor_necrosis}$).

\begin{table}[H]
	\centering
	\begin{tabular}{c|c|c}
		\textbf{Variable (Fixed value)} & $\kappa\;\;\left(5\right)$   &	$\alpha\;\;\left(45\right)$  \\
		\hline
		\textbf{Ranges} & $\left[1,\;10\right]$ & $\left[10,\;100\right]$   \\ 
	\end{tabular}
	\caption{\label{parametros_variables_razon_tumor_necrosis} Variable value parameters.}
\end{table}

Moreover, we take the initial vasculature defined uniformly in space.


\subsubsection{Tumor Ring quotient}

We will start studying the ratio between proliferative tumor density, $\displaystyle\int_{\Omega}T\;dx$ and total tumor density, $\displaystyle\int_{\Omega}\left(T+N\right)\;dx$ and we consider the different values of $\kappa$ and $\alpha$ given in Table $\ref{parametros_variables_razon_tumor_necrosis}$. In fact, we compute the following "ring quotient" (RQ) coefficient:

\begin{equation}\label{RQ}
0\leq\text{RQ}=\dfrac{\displaystyle\int_{\Omega}T\;dx}{\displaystyle\int_{\Omega}\left(T+N\right)\;dx}\leq1.
\end{equation}

Thus, if RQ is near to zero, there exists a high density of necrosis (which implies slim tumor ring) whereas if RQ is close to one, there is not enough necrosis in comparison with proliferative tumor density (which means thick tumor ring). 

\begin{figure}[H]
	\centering
	\begin{subfigure}[b]{0.45\linewidth}
		\includegraphics[width=9cm, height=5cm]{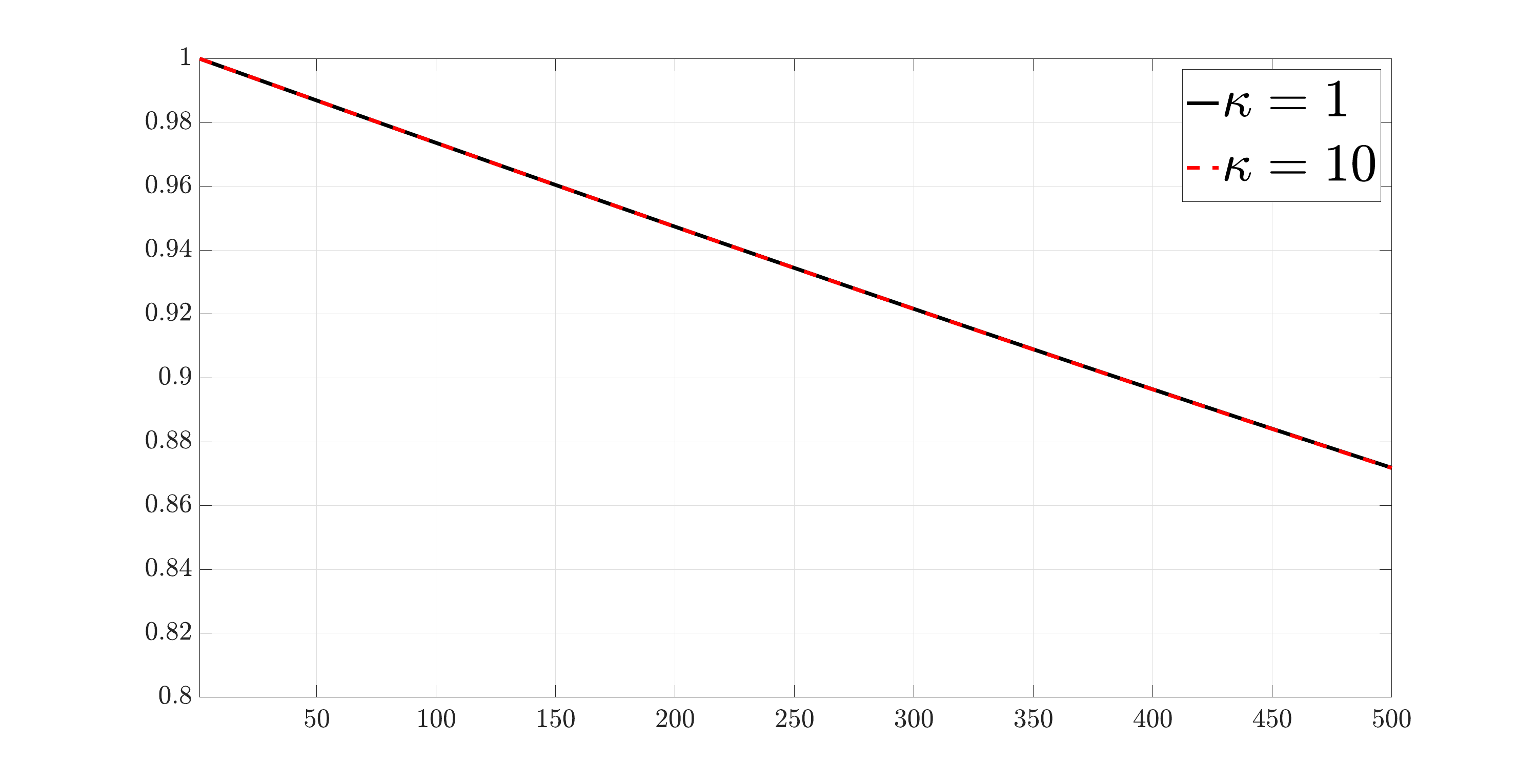}
		\centering
		\caption{RQ versus time for $\kappa$.}
		\label{Ring_dif_kappa}
	\end{subfigure}
	\hspace{1cm}
	\begin{subfigure}[b]{0.45\linewidth}
		\includegraphics[width=9cm, height=5cm]{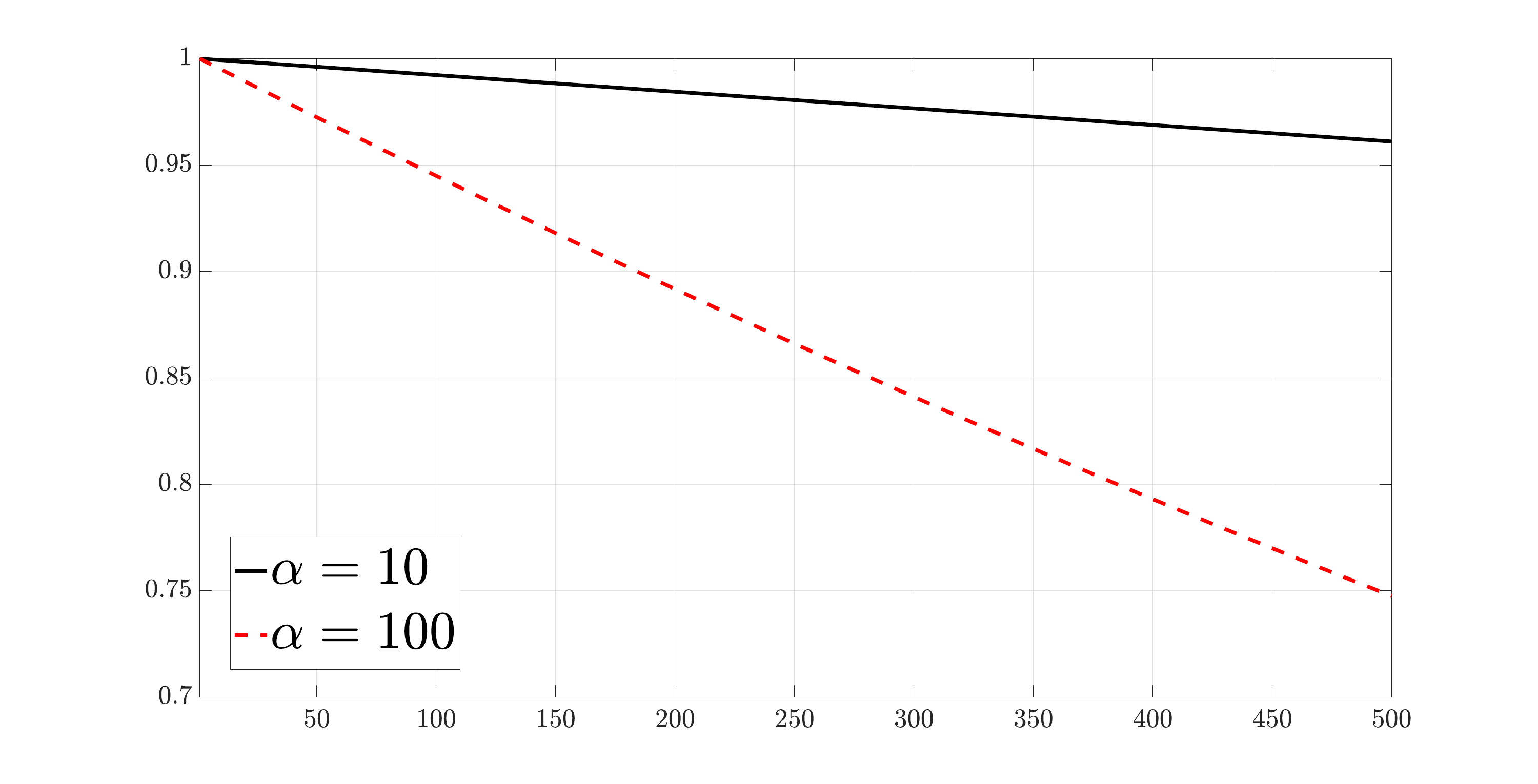}
		\centering
		\caption{RQ versus time for $\alpha$.}
		\label{Ring_dif_alpha}
	\end{subfigure}
	\centering
	\caption{RQ versus time for $\kappa$ and $\alpha$.}
	\label{RQ_kappa1_alpha_beta1}
\end{figure}

We can see in Figs $\ref{Ring_dif_kappa}-\ref{Ring_dif_alpha}$ how the model captures two kinds of tumor ring changing the parameter $\alpha$ and the tumor rings for different $\kappa$ do not change. This means that a change of the rate of tumor destruction for hypoxia produces much difference in the tumor rings.
\\

Hence, the best configurations to obtain a slim (resp. thick) ring would be choose a big (resp. small) $\alpha$.

\subsubsection{Density tumor growth}

In Figure $\ref{densidad_razon_necrosis_tumor}$, we compute the total tumor $\displaystyle\int_{\Omega}\left(T+N\right)\;dx$ for the values of $\kappa$ and $\alpha$ given in Table $\ref{parametros_variables_razon_tumor_necrosis}$.
\begin{figure}[H]
	\centering
	\begin{subfigure}[b]{0.45\linewidth}
			\includegraphics[width=9cm, height=5cm]{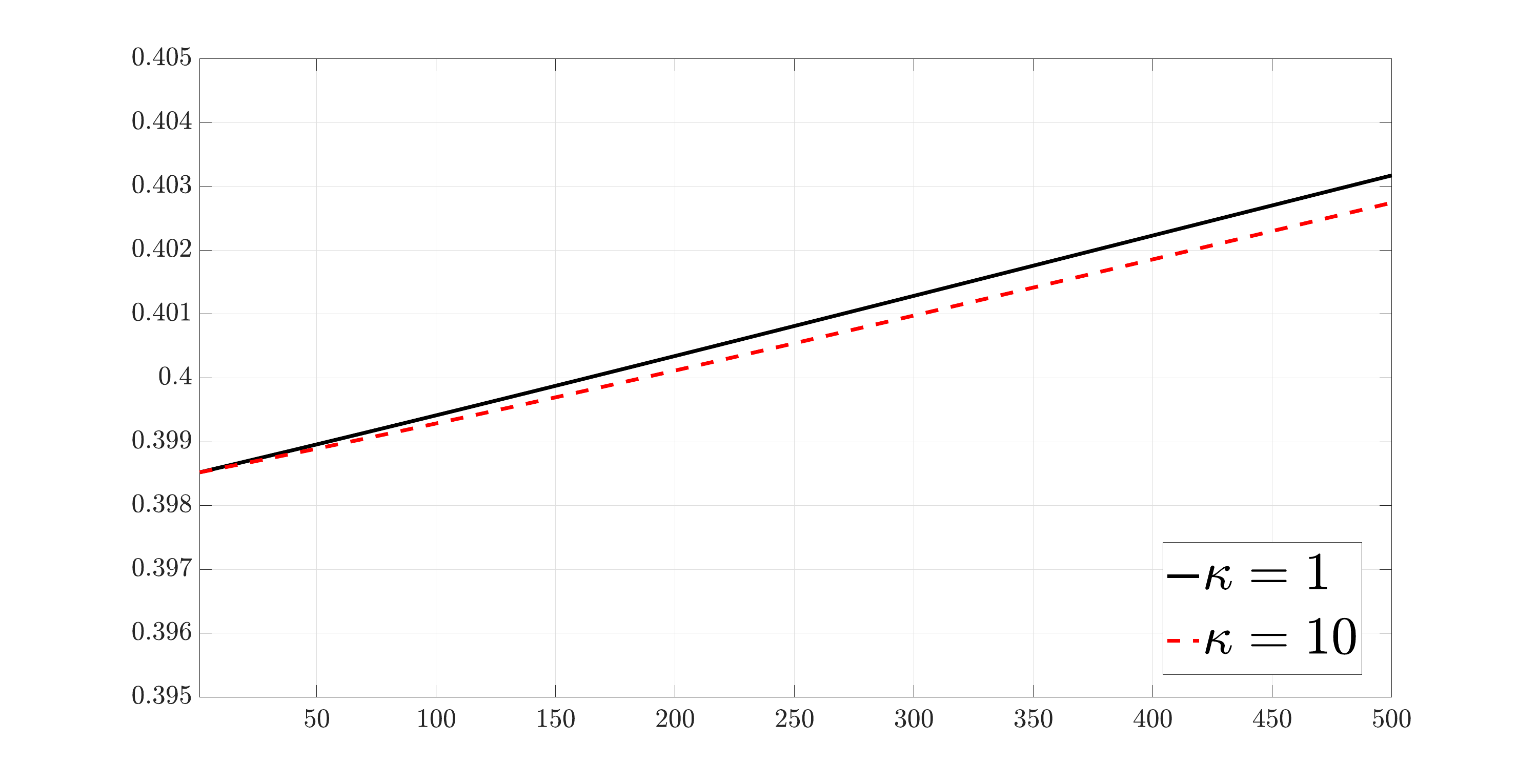}
		\centering
		\caption{$\displaystyle\int_{\Omega}\left(T+N\right)\;dx$ versus time for $\kappa$.}
		\label{densidad_anillo_kappa}
	\end{subfigure}
	\hspace{1cm}
	\begin{subfigure}[b]{0.45\linewidth}
				\includegraphics[width=9cm, height=5cm]{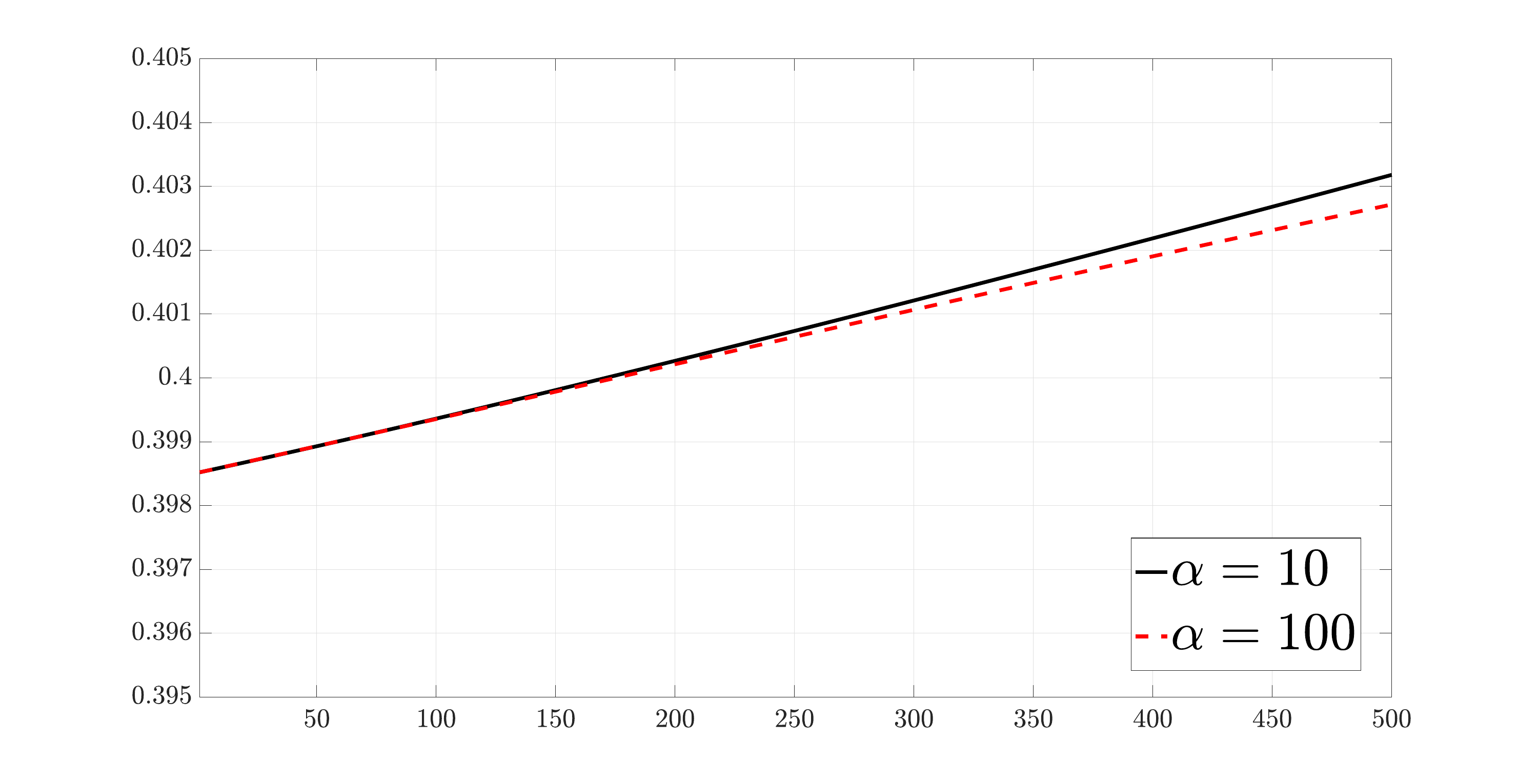}
		\centering
		\caption{$\displaystyle\int_{\Omega}\left(T+N\right)\;dx$ versus time for $\alpha$.}
		\label{densidad_anillo_alpha}
	\end{subfigure}
	\centering
	\caption{$\displaystyle\int_{\Omega}\left(T+N\right)\;dx$ versus time for $\kappa$ and $\alpha$.}
	\label{densidad_razon_necrosis_tumor}
\end{figure}

We can see in Figure $\ref{densidad_razon_necrosis_tumor}$ how the variation in the parameters $\kappa$ and $\alpha$ produces changes in the total tumor density. In fact, the total tumor decreases with respect to $\kappa$ and $\alpha$.
\\

Therefore, we conclude that $\alpha$ is the most important parameter in order to change the tumor ring and both $\kappa$ and $\alpha$ have relevance to change the total density in the tumor growth. 
\subsection{Regularity surface}\label{regularidad}

In this case, we will test if our model $\left(\ref{prob_K_rho2}\right)$ can develop different regularities for the tumor surfaces. Now, we will base our results on the study published in \cite{Victor_2018} where appears the following survival curve:

\begin{figure}[H]
	\includegraphics[width=10cm, height=6cm]{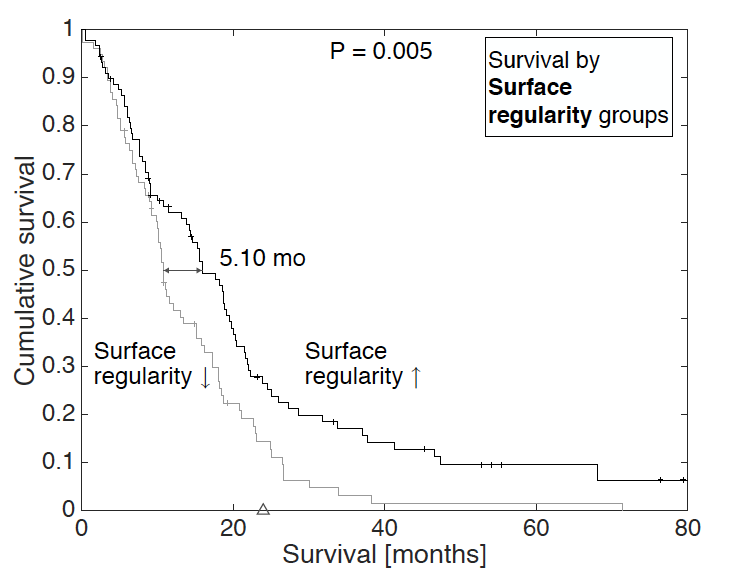}
	\centering
	\caption{Survival vs the regularity surface of GBM.}
	\label{curva_superficie_5}
\end{figure}

From Figure $\ref{curva_superficie_5}$, the authors conclude that tumors with a regular surface have better prognosis than tumors with irregular surface.
\\

Along this Section, we simulate the tumor growth with the initial tumor defined in Figure $\ref{tumor_incial}$, necrosis zero and the vasculature distributed in different zones as in Figure $\ref{vasc_incial}$:
\begin{figure}[H]
	\includegraphics[width=8cm, height=5cm]{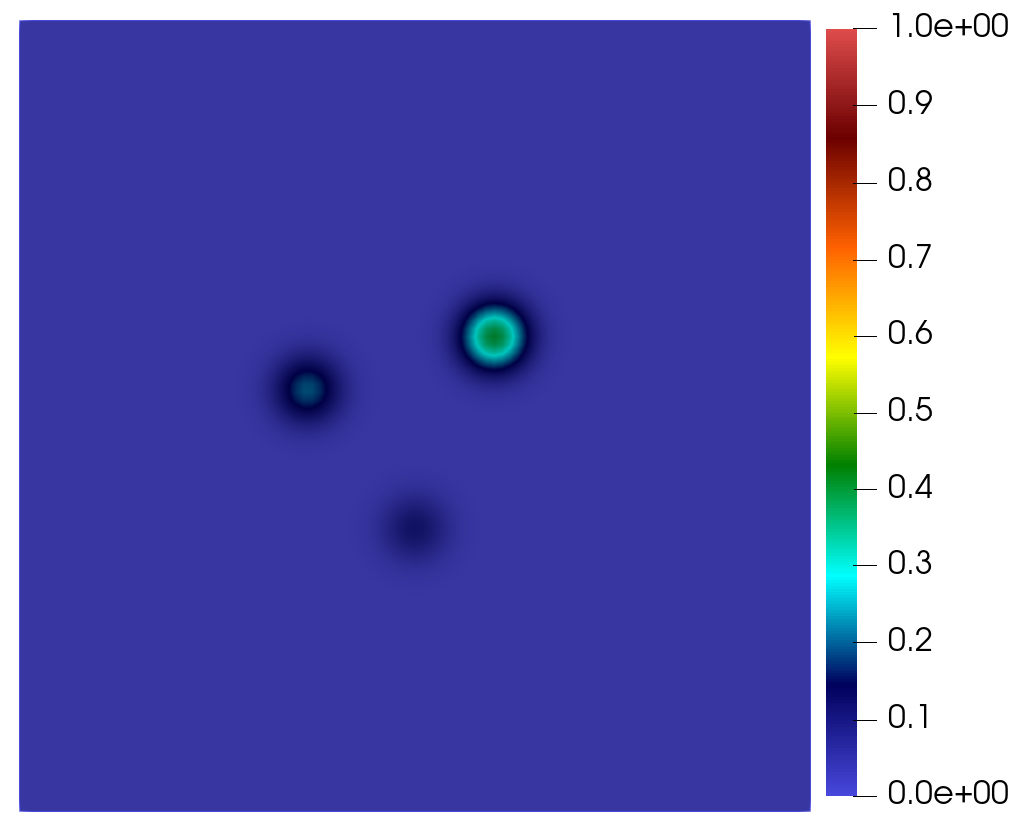}
	\centering
	\caption{Initial vasculature.}
	\label{vasc_incial}
\end{figure}
%
%
Thus, the question is if the chemotaxis term (of tumor going to the vasculature) implies tumor growth with regular or irregular surface. We remember that the chemotaxis term in $\left(\ref{prob_K_rho2}\right)$ is defined by $\kappa\;\nabla\cdot\left(T\;\nabla\Phi\right)$ with $\kappa>0$. 
\\

Now, we want to detect which parameter is more relevant changing the regularity of the tumor surface, showing some simulations in which we move the value of one of them and observe how the tumor changes. For this, it is important the interaction between tumor and vasculature. Then, we will move the parameters which appear in tumor and vasculature equations, $\kappa$, $\alpha$, $\gamma$ and $\delta$. 
\\

For these parameters we take the values of Table $\ref{parametros_variables_crec_reg_irreg}$ (each parameter will change its value in the range, jointly the other parameters take fixed values):

\begin{table}[H]
	\centering
	\begin{tabular}{c|c|c|c|c}
		\textbf{Variable (Fixed value)} &	$\kappa\;\;\left(5\right)$  & 	$\alpha\;\;\left(45\right)$  & 	$\gamma\;\;\left(0.255\right)$  &$\delta\;\;\left(2.55\right)$   \\
		\hline
		\textbf{Ranges}  & $\left[1,\;10\right]$ & $\left[10,\;100\right]$ & 	$\left[0.01,\;0.5\right]$ & $\left[0.1,\;5\right]$  
	\end{tabular}
	\caption{\label{parametros_variables_crec_reg_irreg} Variable value parameters.}
\end{table}

\subsubsection{Regularity Surface quotient}
The pictures of Figure $\ref{SQ_kappa1_alpha_beta1_gamma_delta_beta2_2}$ show the quotient between the area occupied by the total tumor (tumor and necrosis) and the area of ratio the smallest circle containing the tumor. Thus, we present these computations for the different values of $\kappa$, $\alpha$, $\gamma$ and $\delta$ chosen in Table $\ref{parametros_variables_crec_reg_irreg}$. In fact, we compute the following "surface quotient" (SQ) coefficient:

\begin{equation}\label{SQ}
0\leq\text{SQ}=\dfrac{\displaystyle\int_{\Omega}\left(T+N\right)_{\min}\;dx}{\pi\cdot \;\left(\textbf{R}_{\max}\right)^2}\leq 1
\end{equation}
where $\left(T+N\right)_{\min}$ and $\textbf{R}_{\max}$ are defined as follows: 
\begin{equation}\label{Tmin}
\left(T+N\right)_{\min}=\left\{
\begin{array}{ll}
1&\text{if}\;\;T+N\geq 0.001,\\
\\
0&\text{otherwise}. 
\end{array}
\right.
\end{equation}

\begin{equation}\label{Rmax}
\textbf{R}_{\max}=\max\left\{\text{ratio of the subdomain where }\left(T+N\right)_{\min}=1\right\}.
\end{equation}

Thus, we will deduce that if SQ is near to zero, tumor will have an irregular surface whereas if SQ is close to one, tumor will have a regular surface.
\begin{figure}[H]
	\centering
	\begin{subfigure}[b]{0.45\linewidth}
			\includegraphics[width=9cm, height=5cm]{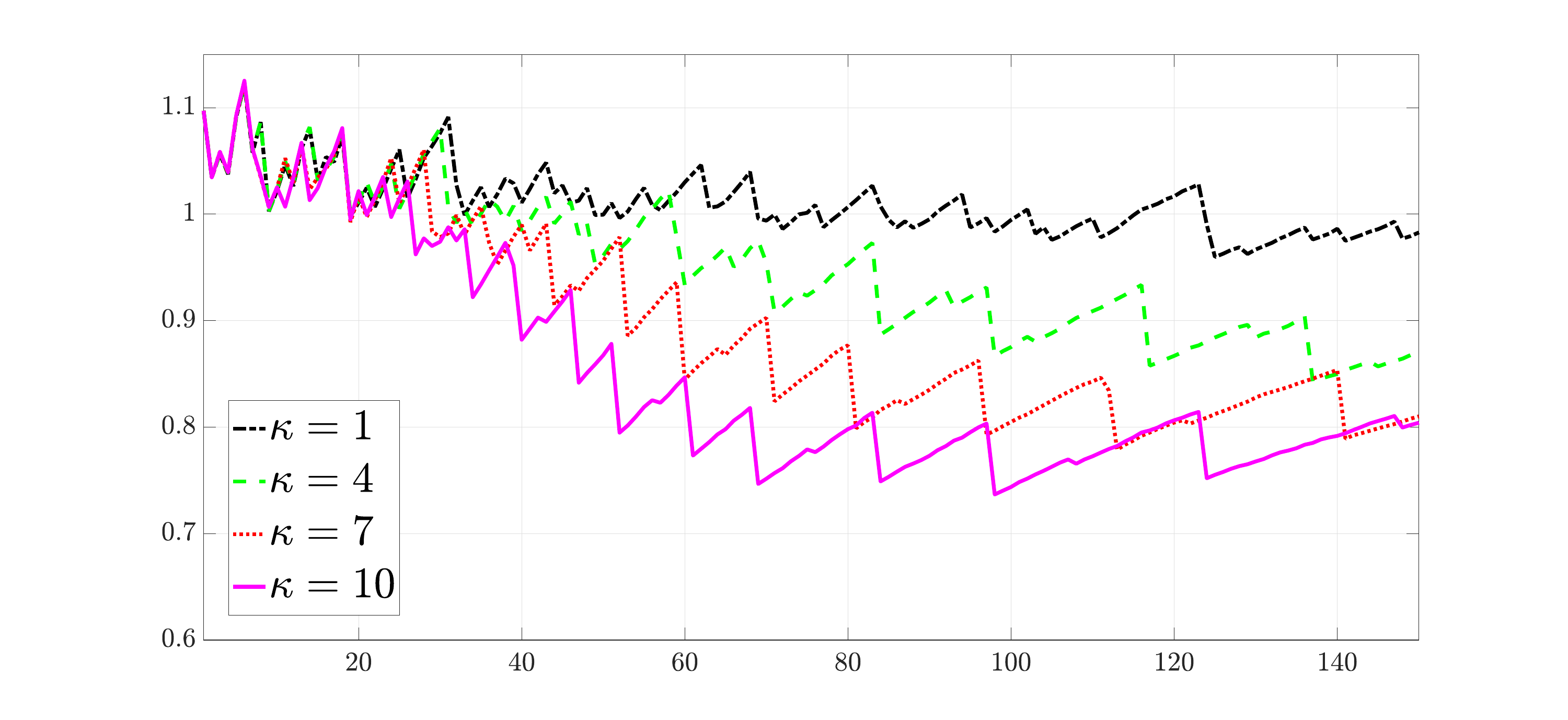}
		\centering
		\caption{SQ versus time for $\kappa$.}
		\label{crec_irreg_kappa}
	\end{subfigure}
	\hspace{1cm}
	\begin{subfigure}[b]{0.45\linewidth}
			\includegraphics[width=9cm, height=5cm]{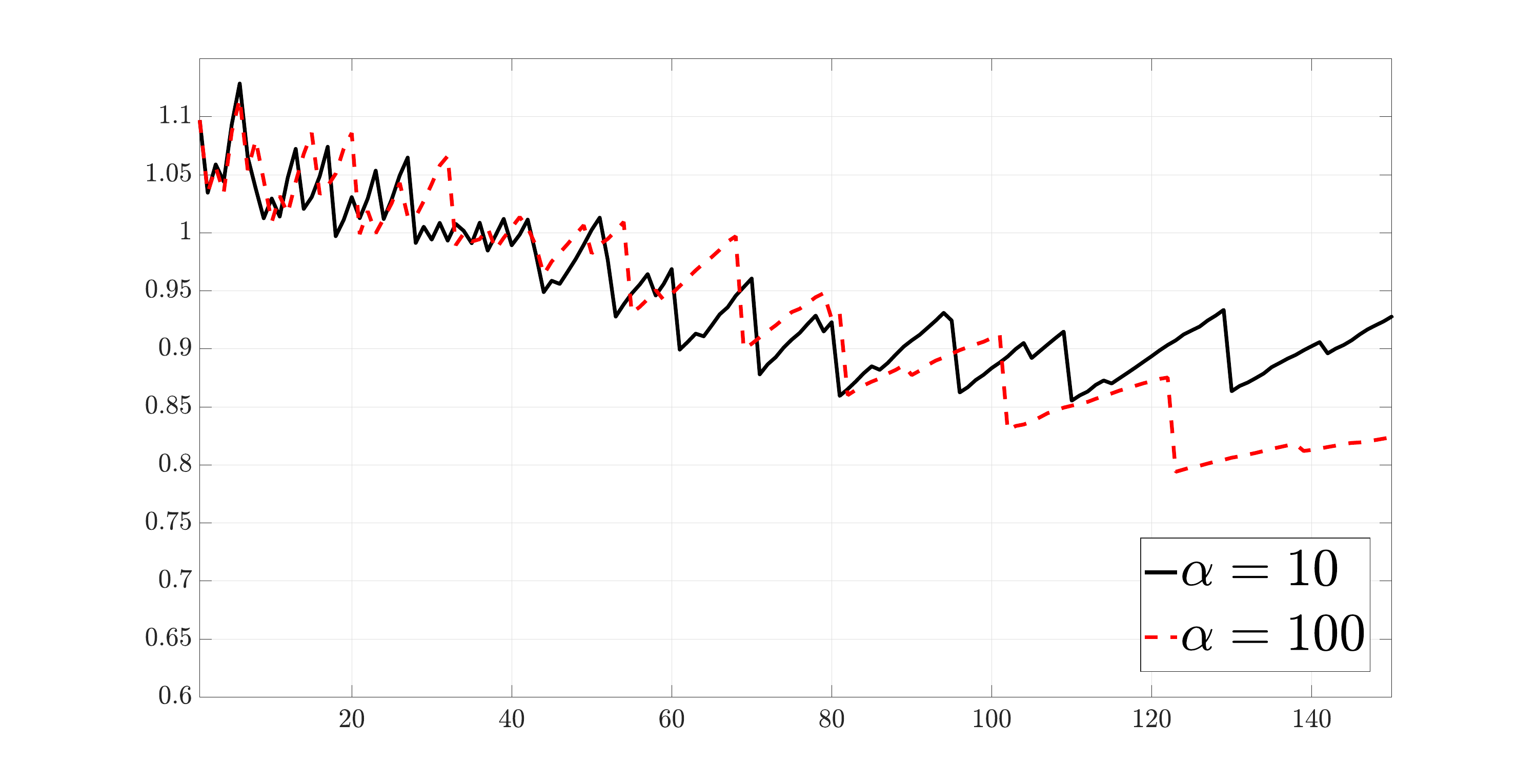}
		\centering
		\caption{SQ versus time for $\alpha$.}
		\label{crec_irreg_alpha}
	\end{subfigure}
\end{figure}
\begin{figure}[H]\ContinuedFloat
		\centering
	\begin{subfigure}[b]{0.45\linewidth}
			\includegraphics[width=9cm, height=5cm]{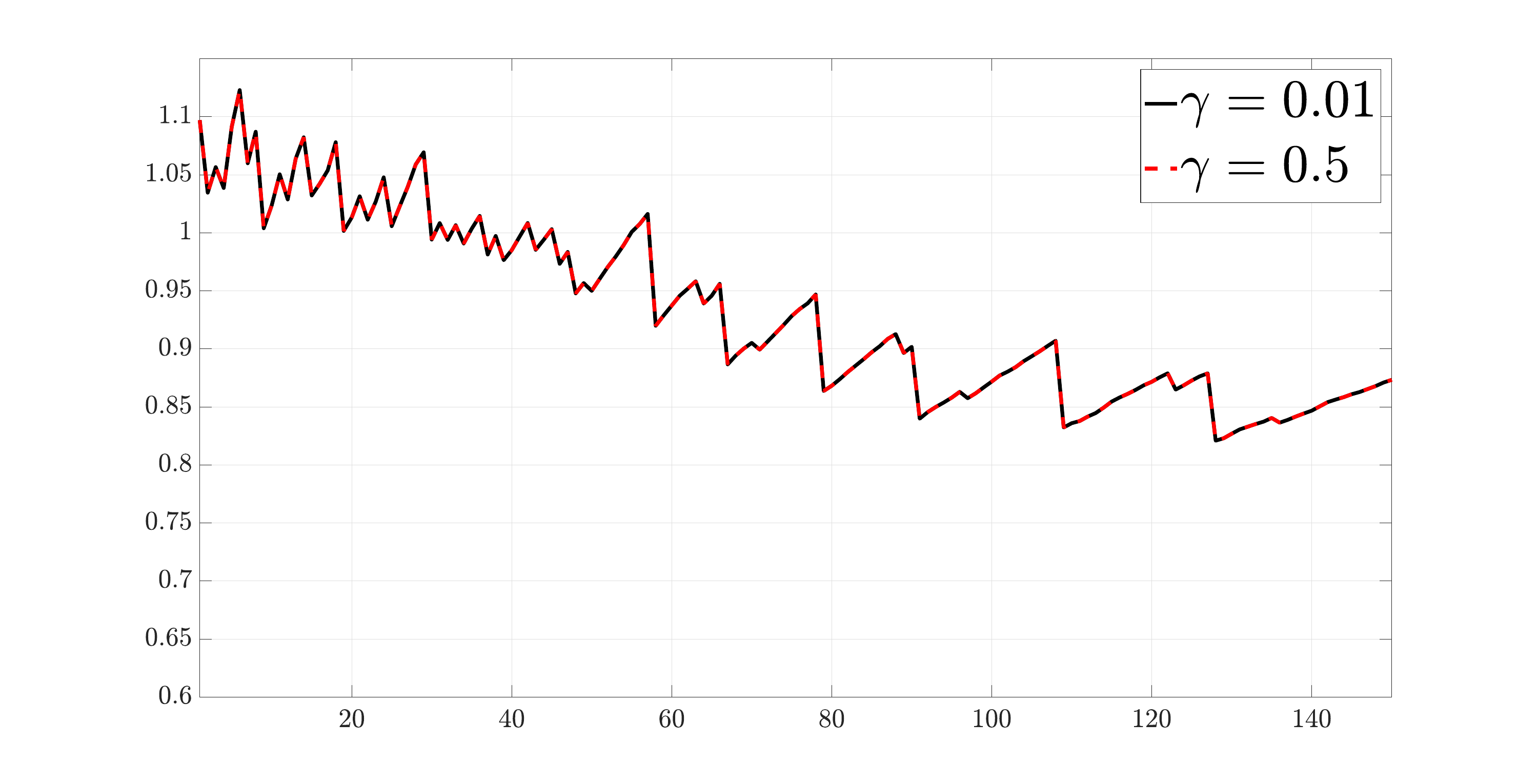}
		\centering
		\caption{SQ versus time for $\gamma$.}
		\label{crec_irreg_gamma}
	\end{subfigure}
\hspace{1cm}
	\begin{subfigure}[b]{0.45\linewidth}
	\includegraphics[width=9cm, height=5cm]{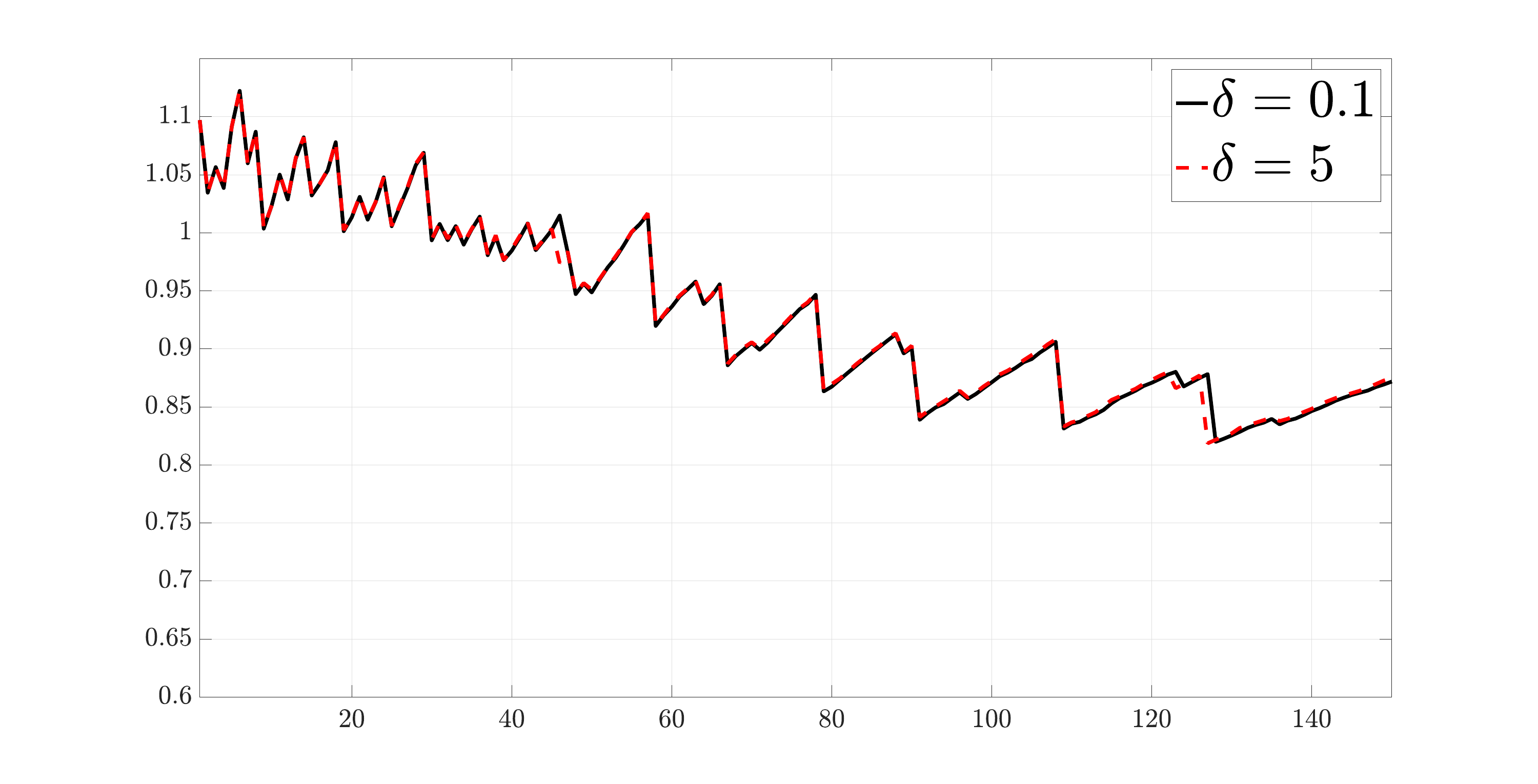}
	\centering
	\caption{SQ versus time for $\delta$.}
	\label{crec_irreg_delta}
\end{subfigure}
	\centering
\caption[]{SQ versus time for $\kappa_1$, $\alpha$, $\gamma$ and $\delta$.}
\label{SQ_kappa1_alpha_beta1_gamma_delta_beta2_2}
\end{figure}
\begin{obs}
	By the size of mesh considered, at the beginning of the pictures given in Figure $\ref{SQ_kappa1_alpha_beta1_gamma_delta_beta2_2}$, the value of SQ is larger than $1$ and it is observed oscillations in the graphs of SQ. Indeed, if we consider a mesh size smaller, these initial values of SQ and the oscillations can be reduced. In order to check this, we show an example of SQ versus time for different $\kappa$ considering a mesh size smaller:
	\begin{figure}[H]
		\centering
					\includegraphics[width=9cm, height=5cm]{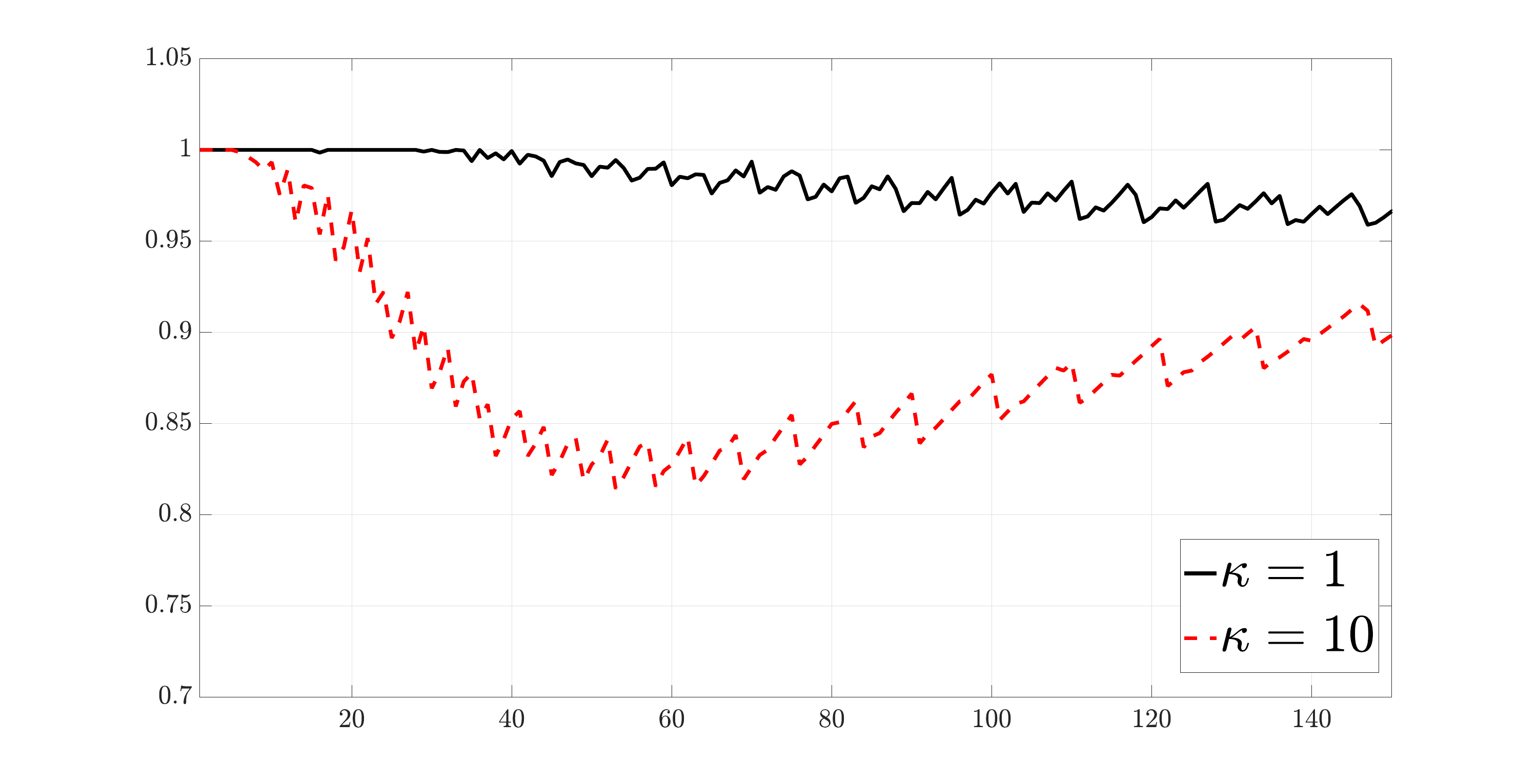}
			\centering
			\caption{SQ versus time for $\kappa$ using a mesh size smaller.}
			\label{crec_irreg_kappa_malla_fina}
	\end{figure}
However, we think that it is not necessary the use of a mesh size smaller since we obtain the same behaviour (in average) for the mesh considered initially and with this mesh, we reduce the computational time. 
\end{obs}
We see in Figs $\ref{crec_irreg_kappa}-\ref{crec_irreg_delta}$ how our model differentiates two kinds of tumor growth changing the parameter $\kappa$, see Figure $\ref{crec_irreg_kappa}$, and with lower variation for $\alpha$, see Figure $\ref{crec_irreg_alpha}$. On the other hand, we do not appreciate changes in the variation of parameters $\gamma$ and $\delta$ for the irregularity of tumor growth as we see in Figs $\ref{crec_irreg_gamma}-\ref{crec_irreg_delta}$.

\subsubsection{Area}

Once we have identified that the more important parameters for the regularity surface are firstly $\kappa$ and later $\alpha$, we measure the area of total tumor for these parameters as in Table $\ref{parametros_variables_crec_reg_irreg}$:

%

\begin{figure}[H]
	\centering
	\begin{subfigure}[b]{0.45\linewidth}
			\includegraphics[width=9cm, height=5cm]{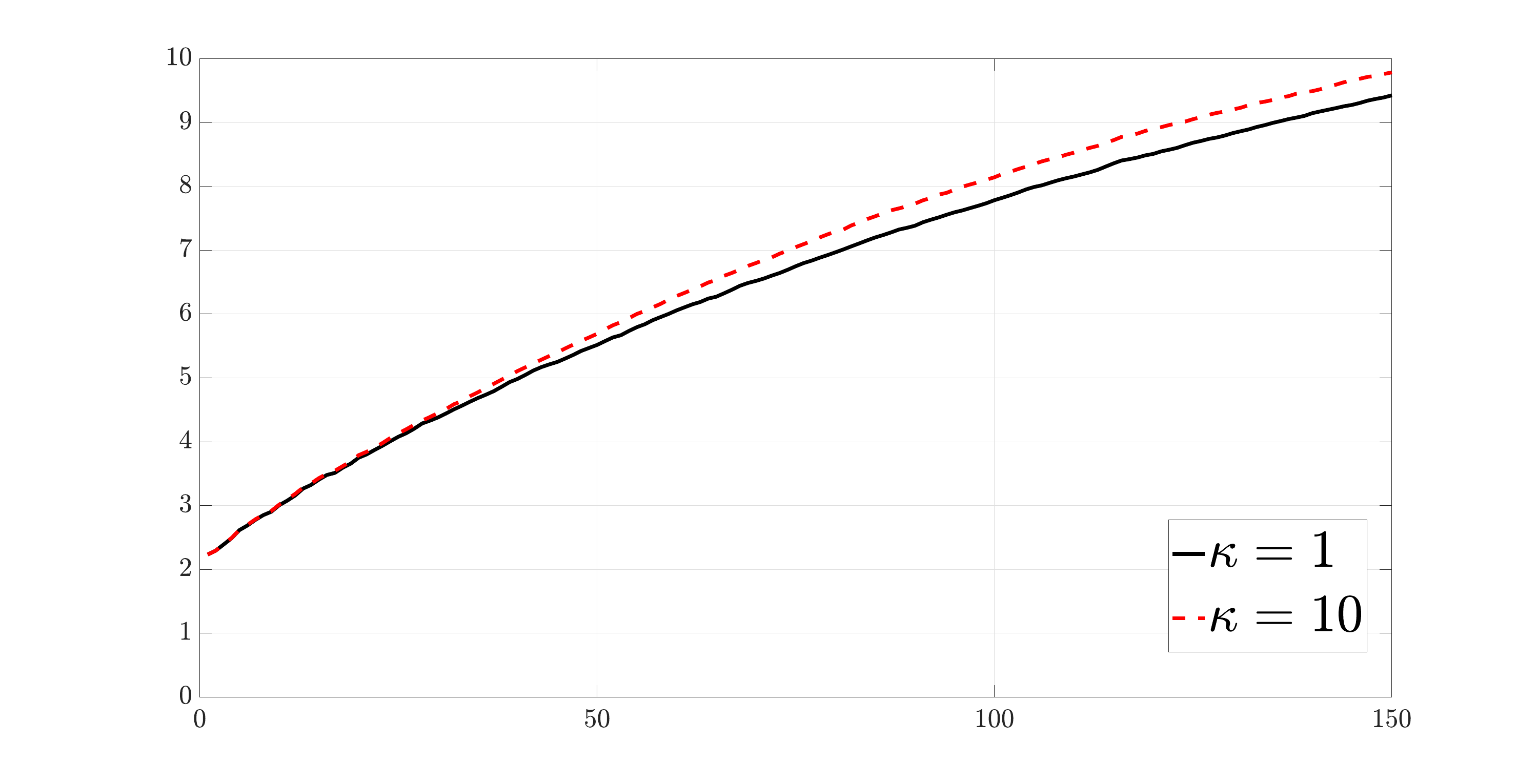}
		\centering
		\caption{Area of total tumor versus time for $\kappa$.}
		\label{area_kappa}
	\end{subfigure}
	\hspace{1cm}
	\begin{subfigure}[b]{0.45\linewidth}
			\includegraphics[width=9cm, height=5cm]{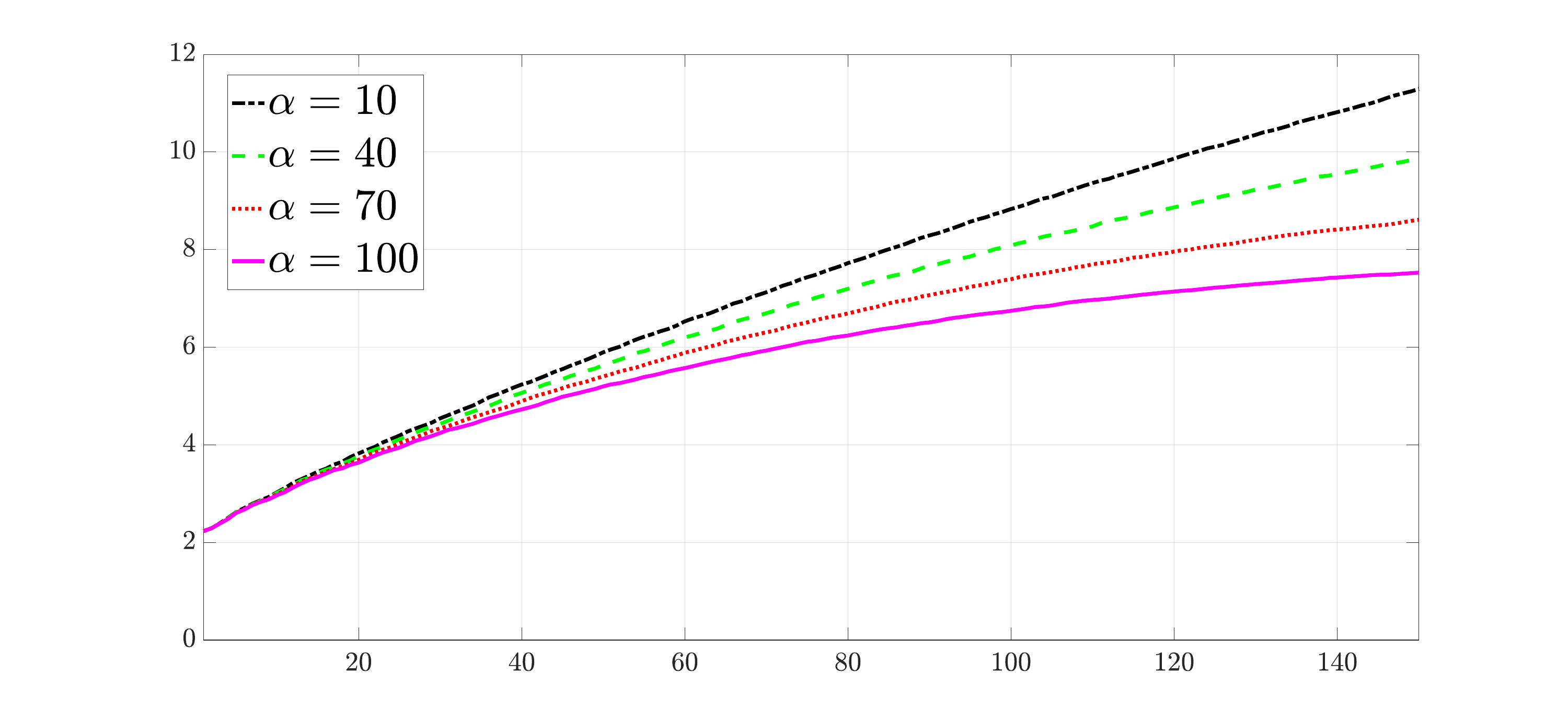}
		\centering
		\caption{Area of total tumor time for $\alpha$.}
		\label{area_alpha}
	\end{subfigure}
	\centering
	\caption{Area of total tumor versus time for $\kappa$ and $\alpha$.}
	\label{area_crec_necrosis_tumor}
\end{figure}

We see in Figure $\ref{area_crec_necrosis_tumor}$ how the largest area corresponds to the smallest $\alpha=10$ and the smallest area holds for the highest $\alpha=100$. In the case of variation of $\kappa$, Figure $\ref{area_kappa}$, a similar influence in the total tumor area for $\kappa=1$ and $\kappa=10$ is observed.
\\

Thus, we have obtained a higher variation of total area for the different values of $\alpha$ than for $\kappa$, see Figure $\ref{area_crec_necrosis_tumor}$. Nevertheless, in the simulation of the "surface quotient" (SQ), we obtained more variation between the different values of $\kappa$ that for the different values of $\alpha$, see Figure $\ref{SQ_kappa1_alpha_beta1_gamma_delta_beta2_2}$. Hence, the factor which modifies this change is $\textbf{R}_{\max}$, defined by $\ref{Rmax}$. In fact, $\textbf{R}_{\max}$ will change more with the variation of $\kappa$ than for the variation of $\alpha$.

\subsubsection{Tumor growth}\label{tumorgrwoth}
Here, we examine the tumor growth for $\kappa=10$ in five times step in order to see the variation in space of tumor. For this growth, the rest of parameters take the fixed values showed in Table $\ref{parametros_variables_crec_reg_irreg}$.
\\

\begin{figure}[H]
	\centering
	\begin{subfigure}[b]{0.15\linewidth}
		\includegraphics[width=1.2\linewidth]{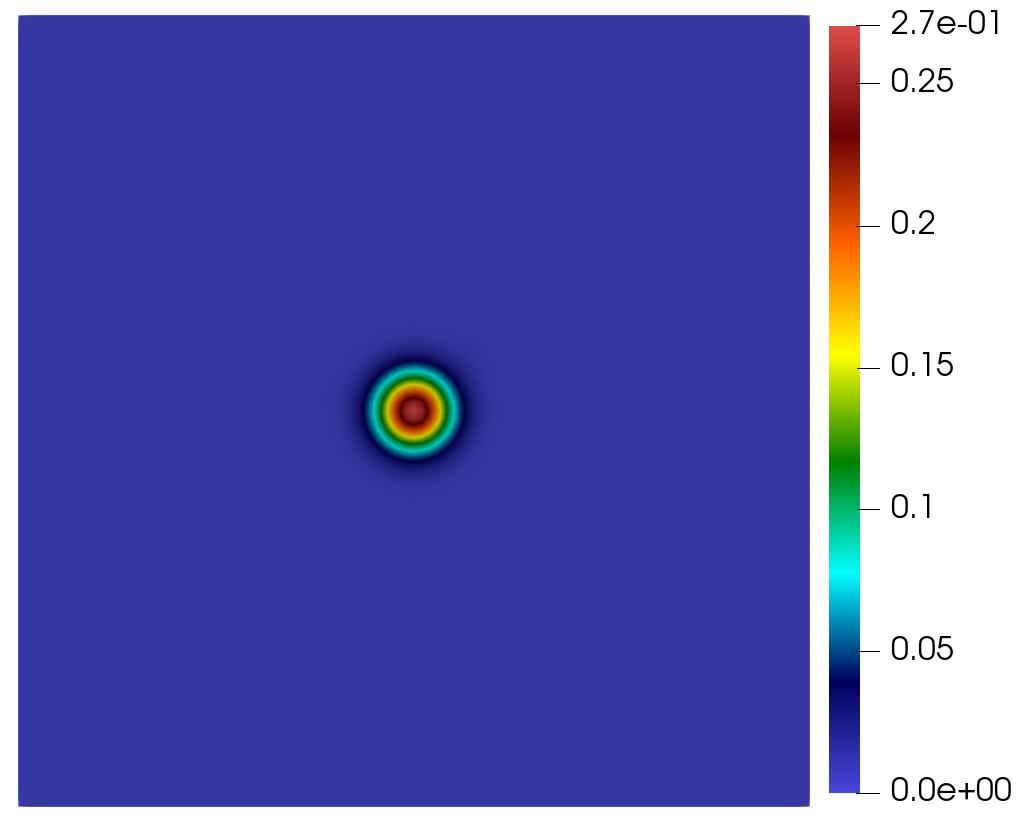}
		\centering
		\caption{$t=50$}
		\label{kappa_100_a}
	\end{subfigure}
	\hspace{0.4cm}
	\begin{subfigure}[b]{0.15\linewidth}
		\includegraphics[width=1.2\linewidth]{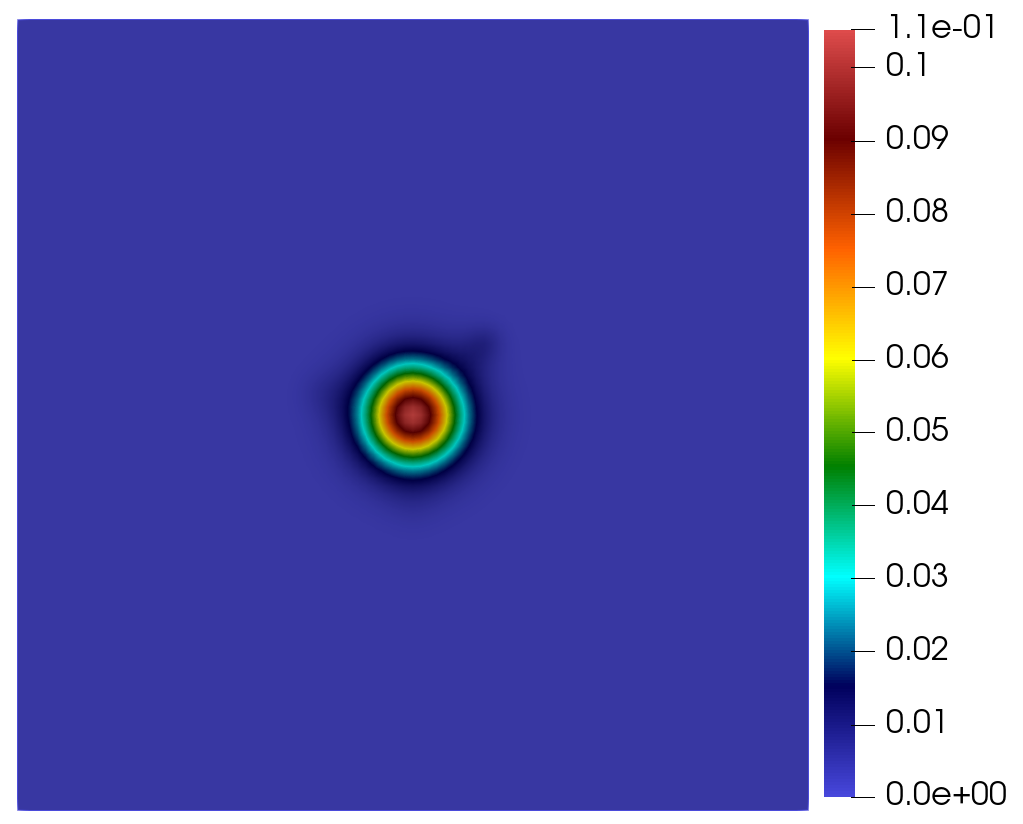}
		\centering
		\caption{$t=100$}
		\label{kappa_100_b}
	\end{subfigure}
	\hspace{0.4cm}
	\begin{subfigure}[b]{0.15\linewidth}
		\includegraphics[width=1.2\linewidth]{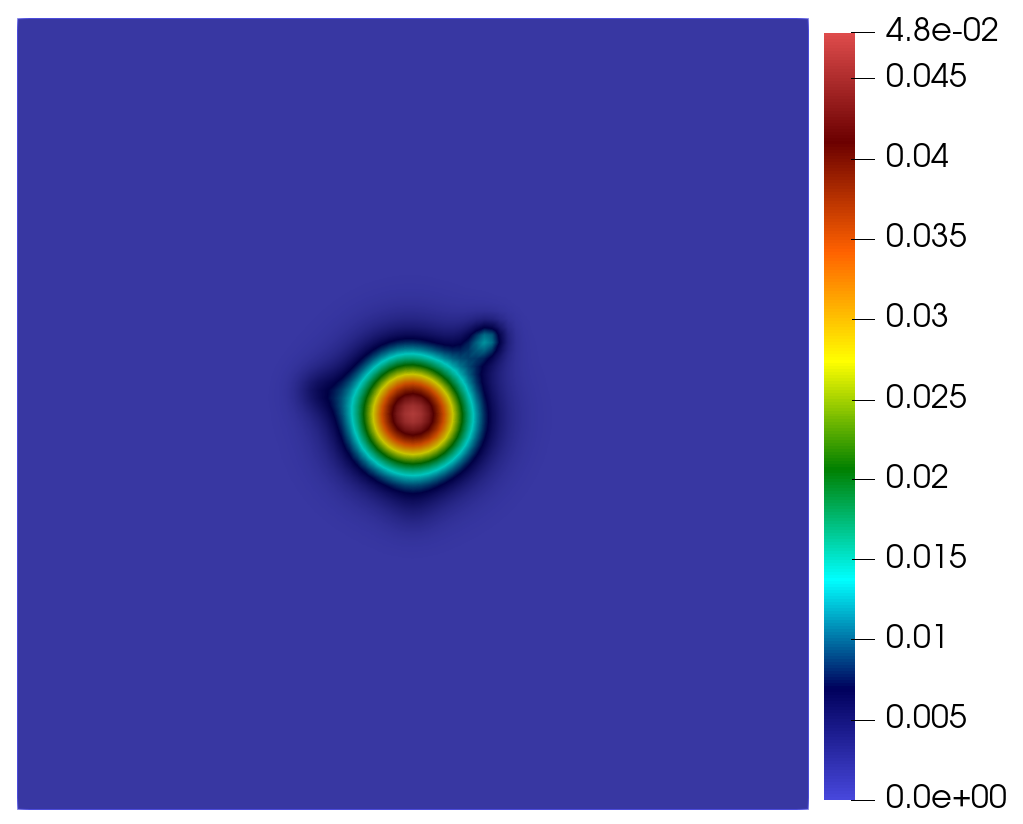}
		\centering
		\caption{$t=150$}
	\end{subfigure}
	\hspace{0.4cm}
	\begin{subfigure}[b]{0.15\linewidth}
		\includegraphics[width=1.2\linewidth]{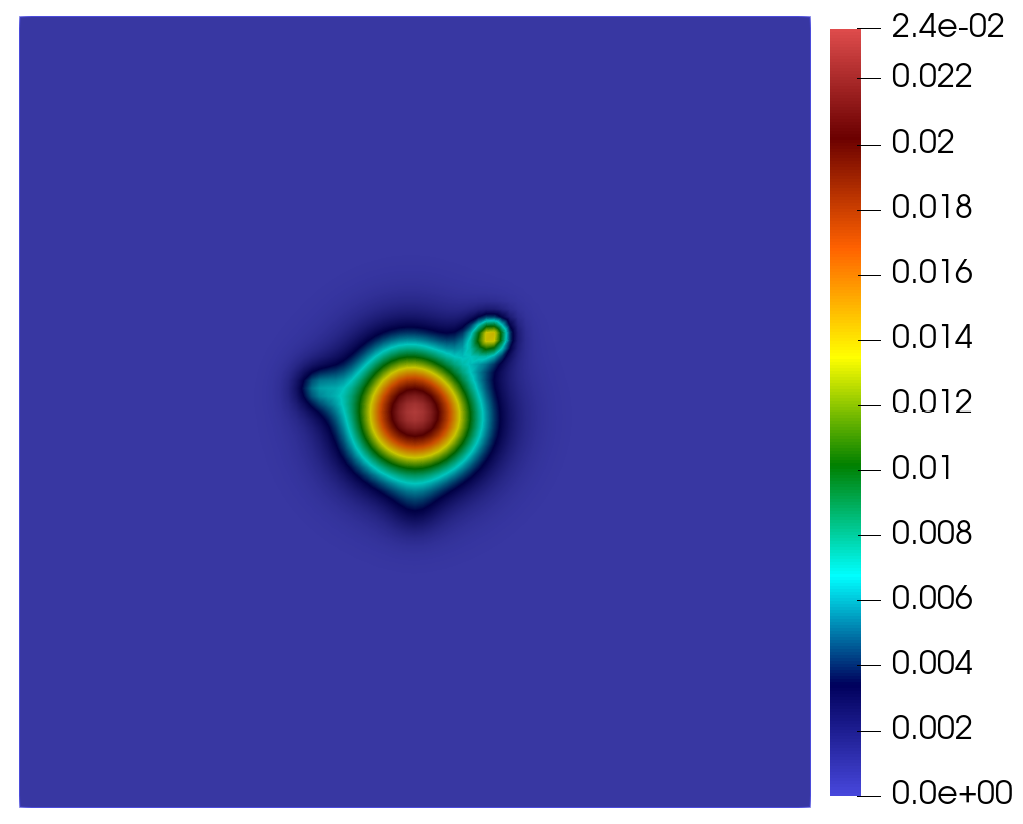}
		\centering
		\caption{$t=200$}
	\end{subfigure}
	\hspace{0.4cm}
	\begin{subfigure}[b]{0.15\linewidth}
		\includegraphics[width=1.2\linewidth]{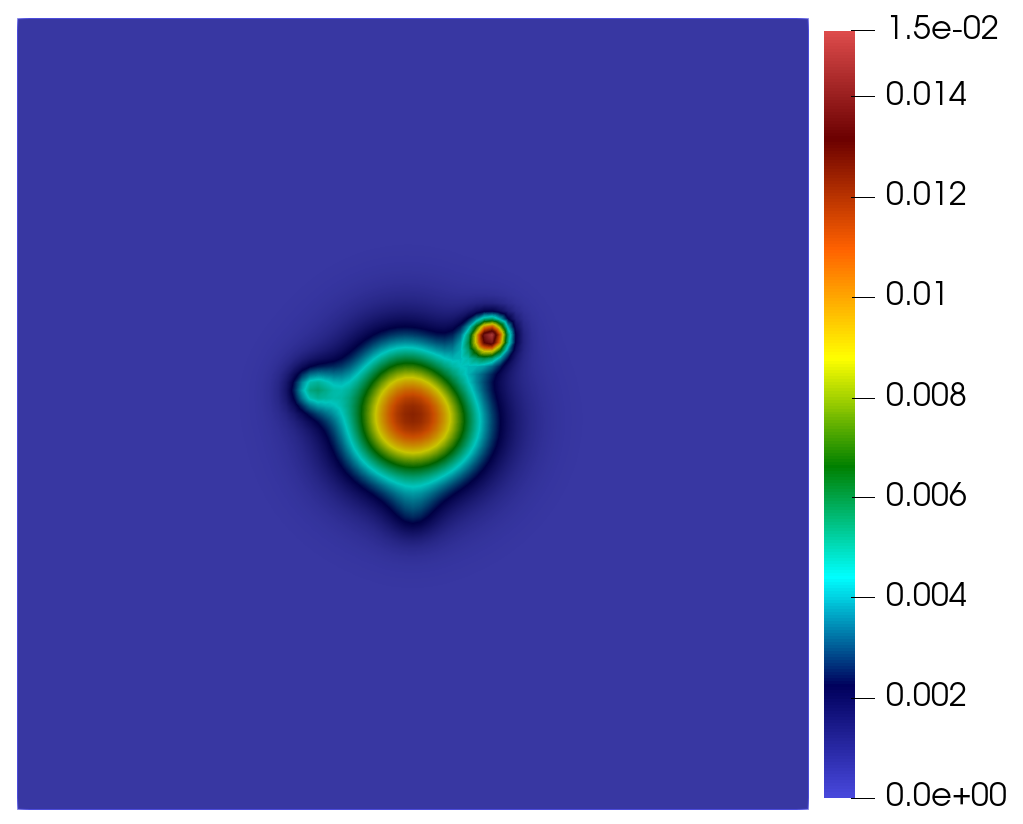}
		\centering
		\caption{$t=250$}
		\label{kappa_100_e}
	\end{subfigure}
	\caption{Irregular tumor growth for $\kappa=10$.}
	\label{crec_irreg_kappa1_100}
\end{figure}

%
We observe an irregular tumor growth for $\kappa=10$ when time increases. These results are in concordance with Figure $\ref{crec_irreg_kappa}$, where we observed a great irregularity for $\kappa=10$ and with Figure $\ref{area_kappa}$, where the area of the tumor for $\kappa=10$ is increasing. 
\\



Finally, we conclude that $\kappa$ is the more relevant parameters in the irregular surface of tumor and $\alpha$ is the most important parameter for total area in the tumor growth.

\subsection{Discussion}\label{conlcusion}

Summarizing the results obtained with respect to the ring width and the regularity surface for the chemotactic and dimensionless system $\left(\ref{prob_K_rho2}\right)$ related to GBM growth model, we deduce that this model can capture these two properties varying some parameters. Moreover, we have proved that the parameters more relevant according to the tumor growth are $\kappa$ and $\alpha$. 
\\


For the tumor ring, where the vasculature is uniformly distributed, the results show that the hypoxia parameter $\alpha$ is the most relevant coefficient as we can observe in Figs $\ref{Ring_dif_kappa}-\ref{Ring_dif_alpha}$. 
\\

In the case of regularity surface, where the vasculature is non-uniformly distributed, the parameter which produces more irregularity in the tumor surface is the chemotaxis parameter $\kappa$, see Figs $\ref{crec_irreg_kappa}-\ref{crec_irreg_delta}$. 
\\

%
%
%
Finally, after the reduction of our model $\left(\ref{probOriginal}\right)$ from $7$ initial parameters to $2$ ($\kappa$ and $\alpha$) which capture the different behaviour of tumor growth, we conclude that hypoxia coefficient $\alpha$ is the main parameter for the tumor ring and area of tumor and $\kappa$ is the most influential parameter for the regularity surface. 

\addcontentsline{toc}{section}{References}
\bibliographystyle{ws-m3as}

\section*{Appendix}\label{apendice}
In this Appendix, we will prove an Alikakos' recursive $L^\infty$ estimate. 
\\

Following the proof of Lemma $\ref{estimaciones3}$, we obtain in $\left(\ref{Linf_prueba7}\right)$ that
\begin{equation}\label{u_Linf_prueba}
\begin{array}{ll}
\displaystyle\max_{t\in (0,T_f)}\|u\|_{L^p\left(\Omega\right)}^p\leq \widetilde{C}\; \max\Big\{\left(p^2+1\right)\;\max_{t\in\left(0,T_f\right)}\|u\|_{L^{p/2}\left(\Omega\right)}^p,\;\|u_0\|_{L^\infty\left(\Omega\right)}^p\Big\}.
\end{array}
\end{equation}

In \cite{Alikakos_1979}, the authors obtained an estimate starting from an estimate like $\left(\ref{u_Linf_prueba}\right)$ but with power $p$ instead of $p^2$. Taking in $\left(\ref{u_Linf_prueba}\right)$ $p=2^k$ for all $k\geq1$, it holds that,

\begin{equation}
\nonumber
\begin{array}{c}
\hspace{2cm}\displaystyle \max_{t\in (0,T_f)}\int_{\Omega}u^{2^k}\;dx\leq C\;\max\Big\{\left(2^{2\;k}+1\right)\;\max_{t\in\left(0,T_f\right)}\left(\int_{\Omega}u^{2^{k-1}}\;dx\right)^2,\;\|u_0\|_{L^{\infty}\left(\Omega\right)}^{2^k}\Big\}\\
\\
\displaystyle\leq C\;C^2\;\max\left\{\left(2^{2\;k}+1\right)\;\left[\max\left\{\left(2^{2\;\left(k-1\right)}+1\right)\;\max_{t\in\left(0,T_f\right)}\left(\int_{\Omega}u^{2^{k-2}}\;dx\right)^2,\right.\right.\right.
\end{array}
\end{equation}
\begin{equation}
\nonumber
\begin{array}{c}
\displaystyle\Big.\left.\left.\|u_0\|_{L^{\infty}\left(\Omega\right)}^{2^{k-1}}\Bigg\}\right]^2,\;\|u_0\|_{L^{\infty}\left(\Omega\right)}^{2^k}\right\}
\displaystyle\\
\\
\leq C\;C^2\;C^{2^2}\;\max\left\{\left(2^{2\;k}+1\right)\;\left(2^{2\;\left(k-1\right)}+1\right)^2\;\left(\max_{t\in\left(0,T_f\right)}\int_{\Omega}u^{2^{k-3}}\;dx\right)^{2^2},\right.\\
\\
\left.\;\left(2^{2\;k}+1\right)\;\|u_0\|_{L^{\infty}\left(\Omega\right)}^{2^k}\right\}\\
\\
\displaystyle\leq C\;C^2\;C^{2^2}\;C^{2^3}\;\max\left\{\left(2^{2\;k}+1\right)\;\left(2^{2\;\left(k-1\right)}+1\right)^2\;\left(2^{2\;\left(k-2\right)}+1\right)^{2^3}\;\max_{t\in\left(0,T_f\right)}\left(\int_{\Omega}u^{2^{k-3}}\;dx\right)^{2^3},\right.
\end{array}
\end{equation}
\begin{equation}\label{Linf_prueba9}
\begin{array}{c}
\displaystyle\left.\left(2^{2\;k}+1\right)\;\left(2^{2\;\left(k-1\right)}+1\right)^2\;\|u_0\|_{L^{\infty}\left(\Omega\right)}^{2^k}
\right\}\leq\ldots\leq\\
\\
\displaystyle\leq\left(C\left(2^{2\;k}+1\right)\right)\;\left(C\left(2^{2\;\left(k-1\right)}+1\right)\right)^2\;\left(C\left(2^{2\;\left(k-2\right)}+1\right)\right)^{2^2}\ldots\left(C\left(2^2+1\right)\right)^{2^{k-1}}\;\widetilde{K}^{2^k}.
\end{array}
\end{equation}
where $\widetilde{K}$ is the constant that dominates $\|u\|_{L^{1}\left(\Omega\right)}$ for all time, since $u\in L^\infty\left(0,T_f;\;L^1\left(\Omega\right)\right)$ (using Lemma $\ref{estimaciones2}$, taking into account that $\|u_0\|_{L^{\infty}\left(\Omega\right)}$ and the hypothesis $\left(\ref{hipotesis0}\right)$). Thus, from $\left(\ref{Linf_prueba9}\right)$

\begin{equation}\label{Linf_prueba10}
\begin{array}{c}
\displaystyle \max_{t\in (0,T_f)}\int_{\Omega}u^{2^k}\;dx\leq
\left(a\;2^{2\;k}\right)\left(a\;2^{2\left(k-1\right)}\right)^2\left(a\;2^{2\left(k-2\right)}\right)^{2^2}\left(a\;2^{2\left(k-3\right)}\right)^{2^3}\ldots\left(a\;2^2\right)^{2^{k-1}}\widetilde{K}^{2^k}.
\end{array}
\end{equation}
for a certain $a\geq3\;C$ since $C\left(2^{2\left(k-j\right)}+1\right)\leq a\;2^{2\;k}$ if $a\geq3\;C$ for all $j=0,\ldots,k-1$. Thus, we can express $\left(\ref{Linf_prueba10}\right)$ as

\begin{equation}\label{Linf_prueba11}
\begin{array}{c}
\displaystyle \max_{t\in (0,T_f)}\int_{\Omega}u^{2^k}\;dx\leq \displaystyle a^{\sum_{j=0}^{k-1}2^j}\;2^{2\;\sum_{j=0}^{k-1}\left(k-j\right)2^j}\;\widetilde{K}^{2^k}=a^{2^k-1}\;2^{\left(-k-6+2^{k+1}\right)}\;\widetilde{K}^{2^k}.
\end{array}
\end{equation}

Taking the limit as $k\rightarrow+\infty$ of the $1/2^k$ power of both sides of $\left(\ref{Linf_prueba11}\right)$ we obtain

\begin{equation}\label{Linf_prueba12}
\begin{array}{c}
\displaystyle \max_{t\in (0,T_f)}\|u\|_{L^\infty\left(\Omega\right)}\leq \lim_{k\to+\infty}\left(a^{\frac{2^k-1}{2^k}}\;2^{\frac{\left(-k-6+2^{k+1}\right)}{2^k}}\;\widetilde{K}\right)=a\;2^2\;\widetilde{K}.
\end{array}
\end{equation}

Hence,
$$u\in L^\infty\left(0,+\infty;\;L^\infty\left(\Omega\right)\right).$$

\end{document}